\newif\ifTwoColumn              % short version aimed at six pages
\newif\ifTechReport

\TechReporttrue  %comment out for IEEE style (applies to 1-column case only)

\ifTwoColumn
        \documentclass[twocolumn,5p,number]{elsarticle} 
\else
                \documentclass[review,3p,number]{elsarticle_arXiv}      
\fi
        
\usepackage{graphicx,color,array,subfigure,hyperref,url}
\usepackage{amsmath,amsthm,amssymb,amsfonts,booktabs,multirow,bm,balance}
\usepackage{caption}
\usepackage[normalem]{ulem}

\theoremstyle{plain}
\newtheorem{theorem}{Theorem}
\newtheorem{lemma}{Lemma}
\newtheorem{proposition}{Proposition}

\newtheorem*{claim*}{Claim}

\newtheorem*{example*}{Example}
\newtheorem{example}{Example}

\theoremstyle{definition}
\newtheorem{definition}{Definition}

\theoremstyle{remark}

\newtheorem*{remark*}{Remark}

\newcommand{\R}{\mathbb{R}}

\newcommand{\bs}{\textbf}

\newcommand{\blue}{\textcolor{black}}

\newcommand{\vol}{\text{vol}}
\newcommand{\rect}{\text{rect}}
\newcommand{\ellipse}{\text{ell}}
\newcommand{\poly}{\text{poly}}

\begin{document}
\begin{frontmatter}

\ifTwoColumn
	\title{Robust Optimal Control with Adjustable Uncertainty Sets}
\else
	\title{Robust Optimal Control with Adjustable Uncertainty Sets \footnote{This manuscript is the preprint of a paper submitted to Automatica and is subject to Elsevier copyright. Elsevier maintains the sole rights of distribution or publication of the work in all forms and media. If accepted, the copy of record will be available at http://www.journals.elsevier.com/automatica/.}}
\fi

\author[ethz]{Xiaojing Zhang}
\author[ethz]{Maryam Kamgarpour}
\author[ethz]{Angelos Georghiou}
\author[oxford]{Paul Goulart}
\author[ethz]{John Lygeros}

\address[ethz]{Automatic Control Laboratory, ETH Zurich, Switzerland}
\address[oxford]{Department of Engineering Science, Oxford University, United Kingdom}

\tnotetext[fninfo]{E-mail addresses: \texttt{xiaozhan@control.ee.ethz.ch} (X. Zhang), \texttt{mkargam@control.ee.ethz.ch} (M. Kamgarpour), \texttt{angelosg@control.ee.ethz.ch} (A. Georghiou), \texttt{paul.goulart@eng.ox.ac.uk} (P. Goulart), \texttt{jlygeros@control.ee.ethz.ch} (J. Lygeros).}

\begin{abstract}	
% \blue{ Robust control design for constrained uncertain systems is a well-studied topic. Given a known uncertainty set, the objective is to find a control policy that minimizes a given cost and satisfies the system's constraints for all possible uncertainty realizations.}
% In this paper, we extend the \blue{standard finite-horizon}  robust optimal control problems by treating the uncertainty sets as additional decision variables. \blue{In } 
\blue{In this paper, we develop a unified framework for  studying constrained robust optimal control problems with \emph{adjustable uncertainty sets}. In contrast to standard constrained  robust optimal control problems with known uncertainty sets, we treat the uncertainty sets in our problems as additional decision variables.} In particular, given a \blue{finite prediction horizon} and a metric for adjusting the uncertainty sets, we address the question of determining the optimal size and shape of the uncertainty sets, while simultaneously ensuring the existence of a control policy that will keep the system within its constraints for all possible disturbance realizations inside the adjusted uncertainty set. Since our problem subsumes the classical constrained robust optimal control design problem, it is computationally intractable in general. We demonstrate that by restricting the families of admissible uncertainty sets and control policies, the problem can be formulated as a tractable convex optimization problem. We show that our framework captures several families of (convex) uncertainty sets of practical interest, and illustrate our approach on a demand response problem of providing control reserves for a power system.

\vspace{0.5cm}

\end{abstract}

\end{frontmatter}

%----------------------------------------------------------------------------------
\section{Introduction}
Robust \blue{finite-horizon} optimal control of constrained linear systems subject to additive uncertainty has been studied extensively in the literature, both in the control %\cite{Kothare1996,  LofbergADF2003, Goulart2006, CalafioreMultiPeriodPortFolioOpt2008} MunozRamirezCamachoAlamo2006, BemporadBorrelliMorari2003
\cite{BemporadMorari1999, LofbergADF2003, SkafBoyd2010} and operations research community \cite{ben2004adjustable, bertsimasThiele2006robustInventory, postekHertog2014multi}. Apart from issues such as stability and recursive feasibility that arise in the context of Model Predictive Control, significant amount of research is concerned with the approximation and efficient computation of the optimal control policies associated with such problems \cite{mayne:rawlings:rao:scokaert:00, SimMPCTAC2010, camacho2013model}.

%\blue{Standard constrained finite-horizon} robust control deals with problems in which the uncertainty sets are known a priori. 
\blue{Commonly, robust control problems of constrained systems over a finite horizon deal with uncertainty sets that are known a priori.}
In this paper, we add another layer of complexity to \blue{these} problems by allowing the uncertainty sets to be decision variables of our problems, and  refer to such problems as \blue{constrained} robust optimal control problems with \emph{adjustable uncertainty sets}. 
For example, if the uncertainty sets are interpreted as a system's \blue{ resilience against disturbance}, then our framework can be used \blue{in a robustness analysis setup} for determining the limits of robustness of a given system. The goal then is to determine the optimal size and shape of the uncertainty sets which maximize a given metric, while  ensuring the existence of a control policy that will keep the system within its constraints. Unfortunately, such problems are computationally intractable in general, since they subsume the standard robust optimal control problem with fixed uncertainty set. The aim of this paper is to propose a systematic method for finding approximate solutions in a computationally efficient way.

%Constrained robust optimal control problems with adjustable uncertainty sets arise in various applications. Consider, for example, the problem of determining the largest disturbance set, according to some metric, such that a system can still be operated within its limits. Such knowledge can use useful in practice, since it allows us to determine if the system and/or controllers are able to reject disturbances from a given set without violating its constraints.
Our work is motivated by reserve provision problems, where the adjustable uncertainty set is interpreted as a \emph{reserve capacity}, which a system can offer to third parties, and for which it receives (financial) reward. In this case, the maximum reserve capacity can be computed by maximizing the size of the uncertainty set. Moreover, the reserve capacity is to be chosen such that for every admissible reserve demand, i.e.\ for every realization within the reserve capacity set, our system is indeed able to provide this reserve without violating its constraints. Reserve provision problems of this kind were first formulated and studied in \cite{zhangCDC2014} where it was shown  that for uncertainty sets described by norm balls, the  problems can be reformulated as tractable convex optimization problems. 
Another problem that admits the interpretation as a robust control problem with adjustable uncertainty set is \emph{robust input tracking} \cite{Vrettos14ifac, GoreckiJonesACC2015}, where the aim is to determine the largest set of inputs that can be tracked by a system without violating its constraints. Reserve provision and input tracking problems have recently received increased attention in demand response applications of control reserves for electrical power grids \cite{VrettosTPS2014, WuRobustWindPowerOpt2014, ZhangECC2015}. %The goal in such cases is to determine the maximum flexibility in electricity consumption, modelled through the uncertainty set, that can be tracked reliably by a electrical system by appropriately changing  its electricity consumption.

The purpose of this paper is two-fold: First, we generalize the work of \blue{\cite{zhangCDC2014,GoreckiJonesACC2015}} by considering a larger class of adjustable uncertainty sets \blue{based on techniques of conic convex optimization}. Second, we provide a unified framework for studying reserve provision, input tracking \blue{and robustness analysis} problems under the umbrella of constrained robust optimal control with adjustable uncertainty sets. The main contributions of this paper with respect to the existing literature can be summarized as follows: 

\begin{itemize}
	\item We show that if $(i)$ the uncertainty sets are restricted to those that can be expressed as affine transformations of properly selected \emph{primitive} convex sets, and $(ii)$ the control policies are restricted to be affine with respect to the elements in these \blue{primitive} sets, then the problems admit convex reformulations that can be solved efficiently, \blue{and whose size grows polynomially in the decision parameters}. \blue{In particular, we extend the results of \cite{zhangCDC2014,GoreckiJonesACC2015} in two ways: First, we show that any convex set can be used as a primitive set, allowing us to target a much larger class of uncertainty sets. Second, by allowing the primitive sets to be defined on higher dimensional spaces than those of the uncertainty sets, we are able to design more flexible uncertainty sets.}
	
	\item We identify families of uncertainty sets of practical interest, including norm-balls, ellipsoids and hyper-rectangles, and show that they can be adjusted efficiently. \blue{Extending the work of \cite{zhangCDC2014,GoreckiJonesACC2015}, we also show that by choosing the primitive set as the simplex, our framework enables us to efficiently optimize over compact polytopes with a predefined number of vertices}. Furthermore, we prove that if the primitive sets are polytopes (e.g.\ the simplex), then our policy approximation gives rise to continuous piece-wise affine controllers.
	
	\item We study a reserve provision problem that arises in power systems, and show how it can be formulated as a robust optimal control problem with an adjustable uncertainty set. The problem is addressed using the developed tools, \blue{and we show that it can be formulated as a linear optimization problem of modest size} that can be solved efficiently within 0.3 seconds, \blue{making it also practically applicable}. \blue{In additional to results available in the literature \cite{GoreckiJonesACC2015, VrettosTPS2014}}, we verify that our method produces the optimal solution for this problem for prediction horizons of up to eight time \blue{steps}. \blue{For larger horizons, we show that the (relative) suboptimality gap is bounded by $0.71\%$.}
		
\end{itemize}

This paper is organized as follows: Section \ref{sec:problemDescription} introduces the general problem setup. Section \ref{sec:modUncerSet} focuses on the problem of adjusting the uncertainty sets, while in Section \ref{sec:decisionRuleApprox}, we return to the original problem and restrict the family of control policies to obtain tractable instances thereof. Section \ref{sec:example} illustrates our approach on a demand response problem, \blue{while Section~\ref{sec:example2} demonstrates the usefulness of allowing the uncertainty sets to be projections of primitive convex sets.} Finally,
 Section \ref{sec:conclusion} concludes the paper. The Appendix contains auxiliary results needed to prove the main results of the paper.

%\subsection*{Related Work}
%Maximization of the volume of convex sets, even for polytopes, is known to be NP-hard in general, so that most research has focused on volume maximization of specific families of sets, such as balls, boxes, or ellipsoids. For these sets, maximizing the volume can be obtained by maximizing the determinant of the matrix describing these sets. As shown in \cite{VandenbergheDetMax1998}, these problems are tractable and convex, thus can be solved efficiently. Similarly, these algorithms can be used to compute the maximum volume ellipsoid in a polyhedron, as well as the minimum volume ellipsoid covering a polytope \cite{boyd2004convex}. Similarly, the authors in \cite{ZhenHertogMaxVolume2015} show that the maximum volume inscribed ellipsoid can be approximated efficiently. 

\subsection*{Notation}
For given matrices $(A_1,\ldots,A_n)$, we define $A:=\text{diag}(A_1,\ldots,A_n)$ as the block-diagonal matrix with elements $(A_1,\ldots,A_n)$ on its diagonal. $A_{ij}$ denotes the \mbox{$(i,j)$-th} element of the matrix $A$, while $A_{\cdot j}$ denotes the $j$-th column of $A$. Given a cone $\mathcal{K}\subseteq\R^l$ and two vectors $a,b\in\R^{l}$, $a\preceq_\mathcal{K} b$ implies $(b-a)\in\mathcal{K}$. For a matrix $B\in\R^{m\times l}$, $B\succeq_\mathcal{K}0$ denotes row-wise inclusion in $\mathcal{K}$. For a symmetric matrix $C\in\R^{n\times n}$, $C\succeq0$ denotes positive semi-definiteness of $C$.  Given vectors $(v_1,\ldots,v_m)$, $v_i\in\R^l$, we denote their convex hull as $\text{conv}(v_1,\ldots,v_m)$. \blue{Moreover, $[v_1,\ldots,v_m] := [v_1^\top\ \ldots\ v_m^\top]^\top \in\R^{lm}$ denotes their vector concatenation.}

\section{Problem Formulation}\label{sec:problemDescription}
In this section, we formulate the robust optimal control problem with adjustable uncertainty set.  We consider uncertain linear systems  of the form
\begin{equation}\label{eq:systemDynamics_generic}
        x_{k+1} = A x_k + B u_k + Ew_k,
\end{equation}
where $x_k\in\R^{n_x}$ is the state at time step $k$ given an initial state $x_0\in\R^{n_x}$, $u_k\in\R^{n_u}$ is the control input and $w_k\in\mathbb{W}_k\subseteq\R^{n_w}$ is an uncertain disturbance. We consider compact polytopic state and input constraints
\begin{equation}\label{eq:XU_def}
\hspace{-0.3cm}
	\begin{array}{ll}
		& x_k\in\mathbb{X} :=\{x\in\R^{n_x}: F_x x \leq f_x\},\  k=1,\ldots,N, \\
		& u_k\in\mathbb{U}:=\{u\in\R^{n_u}: F_u u\leq f_u\},\ k=0,\ldots,N-1,
	\end{array}
\hspace{-0.4cm}
\end{equation}
where $F_x\in\R^{n_f\times n_x},\ f_x\in\R^{n_f},\ F_u\in\R^{n_g\times n_u},\ f_u\in\R^{n_g}$, and $n_f\ (n_g)$ is the number of state (input) constraints. Given a planning horizon $N$, we denote by $\phi_{k}(\mathbf{u},\mathbf{w})$ the predicted state after $k$ time steps resulting from the input sequence $\mathbf{u}:=[u_0,\ldots,u_{N-1}]\in\R^{Nn_u}$ and disturbance sequence $\mathbf{w}:=[w_0,\ldots,w_{N-1}]\in\R^{Nn_w}$. %We assume that $\mathbf{w}$ can be measured directly.% \red{[extension, future work?]}

In contrast to  classical robust control problem formulations, we assume that the uncertainty set $\mathbb{W}_k$ is not fixed and needs to be adjusted according to some objective function $\varrho:\mathcal{P}(\R^{n_w})\to\R$, where $\mathcal{P}(\R^{n_w})$ denotes the power set of $\R^{n_w}$. For example, we may think of $\varrho(\mathbb{W}_k)$ as the volume of $\mathbb{W}_k$, although depending on the application, it can represent other qualities such as the diameter or circumference of $\mathbb{W}_k$. Our objective is to maximize $\varrho(\mathbb{W}_k)$, while simultaneously minimizing some operating cost and ensuring satisfaction of input and state constraints. Hence, the cost to be minimized is given by 
\begin{equation}\label{eq:TotalCost}
         \max_{\mathbf{w}\in\mathcal{W}} \left\{J(\mathbf{u},\mathbf{w})\right\} - \lambda\sum_{k=0}^{N-1}\varrho(\mathbb{W}_k),
\end{equation}
where $ J(\mathbf{u},\mathbf{w}) := \ell_f (\phi_N(\mathbf{u},\mathbf{w})) + \sum_{k=0}^{N-1} \ell \left( \phi_{k+1}(\mathbf{u},\mathbf{w}),u_k \right)$ is some ``nominal" cost function with $\ell:\R^{n_x}\times\R^{n_u}\to\R$ and $\ell_f:\R^{n_x}\to\R$ linear, and $\lambda\geq0$ is a user-defined weighting factor. \blue{Note that convex quadratic cost can also be incorporated in our framework by taking the certainty-equivalent cost $J(\mathbf u,\mathbf{\bar w})$ instead of the min-max cost in \eqref{eq:TotalCost}, where $\mathbf{\bar w}$ is some fixed (or expected) uncertainty.} Due to the presence of the uncertainties $\mathbf{w}$, we consider the design of a causal disturbance feedback policy $\bm\pi(\cdot) := [\pi_0(\cdot),\ldots,\pi_{N-1}(\cdot)]$, with each $\pi_k: \mathbb{W}_0\times\ldots\times\mathbb{W}_{k} \to\R^{n_u}$, such that the control input at each time step is given by $u_k = \pi_k(w_0,\ldots,w_k)$\footnote{Strictly causal policies can be incorporated by restricting $\pi_k(\cdot)$ to depend on $(w_0,\ldots,w_{k-1})$ only. For simplicity, this paper considers causal policies. However, all  subsequent results apply to strictly causal policies with minor modifications.}. Combining \eqref{eq:systemDynamics_generic}--\eqref{eq:TotalCost}, we  express the optimal control problem compactly as
\begin{equation}\label{eq:OptProblem_detGeneral}
\begin{array}{ll}
\text{min}   & \displaystyle\max_{\mathbf{w}\in\mathcal{W}} \left\{\mathbf{c}^\top \bm\pi(\mathbf{w})\right\} - \lambda\bm\varrho(\mathcal{W}) \\[0.3ex]
\text{s.t.}      &  \bm\pi(\cdot)\in\mathcal{C},\ \mathcal{W}\in\mathcal P(\R^{Nn_w}),   \\[0.3ex]
& \mathbf{C}\bm\pi(\mathbf{w}) + \mathbf{D}\mathbf{w} \leq \mathbf{d},\quad \forall\mathbf{w}\in\mathcal{W},
\end{array}
\end{equation}
with decision variables $(\bm\pi(\cdot),\mathcal{W})$, and $\mathcal{W} := \mathbb{W}_0 \times \ldots \times \mathbb{W}_{N-1}$,  $\bm{\varrho}(\mathcal{W}):=\sum_{k=0}^{N-1}\varrho(\mathbb{W}_k)$, and $\mathcal{C} := \{[\pi_0(\cdot),\ldots,\pi_{N-1}(\cdot)]:\ \pi_{k}:\mathbb{W}_0\times\cdots\times\mathbb{W}_k\rightarrow\R^{n_u},\,k=0,\ldots,N-1\}$ is the space of all causal control policies. $\mathbf{c}\in\R^{Nn_u}$, $\mathbf{C}\in\R^{N(n_f+n_g) \times Nn_u}$, $\mathbf{D}\in\R^{N(n_f+n_g) \times Nn_w}$, and $\mathbf{d}\in\R^{N(n_f+n_g)}$ are matrices constructed from the problem data, see e.g.\ \cite{Goulart2006} for an example of such a construction.
%finding  control policies $\bm\pi$ and uncertainty sets $\mathcal{W}=\mathbb{W}_0 \times \ldots \times \mathbb{W}_{N-1}$ that solve the following finite horizon optimal control problem (FHOCP):
%\begin{align*}%\label{eq:OptProblem_general}
%       \min_{\bm\pi,\mathcal{W}} \quad &  Q\left[ J(\bm\pi(\mathbf{w},\bm\delta),\mathbf{w},\bm\delta) \right] - \lambda\sum_{k=0}^{N-1}\varrho(\mathbb{W}_k) \\
%       \text{s.t.} \quad & \forall k\in\{0,\ldots,N-1\}\ \forall\mathbf{w}\in\mathcal{W}\ \nonumber \\
%                               & \text{(i) }\quad \pi_{k}(\mathbf{w},\bm\delta) \in  \mathbb{U}, \quad \forall\bm\delta\in\Delta \nonumber \\ 
%                               & \text{(ii)}\quad \bbP[\phi_{k+1}(\bm{\pi},\mathbf{w},\bm\delta) \in \mathbb{X}] \geq 1-\epsilon. \nonumber
%\end{align*}
%\begin{align}\label{eq:OptProblem_detGeneral}
%       \tag{$\mathcal{P}$}
%       \min_{\bm\pi,\mathcal{W}} \quad &  \quad \max_{\mathbf{w}\in\mathcal{W}} \left\{J(\bm\pi,\mathbf{w}) \right\} - \lambda\sum_{k=0}^{N-1}\varrho(\mathbb{W}_k) \\
%       \text{s.t.} \quad & \quad \forall k\in\{0,\ldots,N-1\}\  \nonumber \\
%                               & \left. \begin{array}{ll}
%                               \text{(i) } & \pi_{k}(\mathbf{w}) \in  \mathbb{U},  \nonumber \\ 
%                                \text{(ii)} & \phi_{k+1}(\bm{\pi},\mathbf{w}) \in \mathbb{X} \nonumber
%                               \end{array} \right\}
%                               \forall\mathbf{w}\in\mathcal{W}.
%\end{align}
We call \eqref{eq:OptProblem_detGeneral} a \emph{robust optimal control problem with adjustable uncertainty set}, since the uncertainty set $\mathcal{W}$ is a decision variable. Note that if $\mathcal{W}$ is given and fixed, then the above problem is a standard robust optimal control problem with additive uncertainty, see e.g.\ \cite{Goulart2006, Lofberg2003} and the references therein. 

Problem \eqref{eq:OptProblem_detGeneral} is intractable because $(i)$ the optimization of the uncertainty set $\mathcal{W}$ is performed over arbitrary subsets  of $\R^{Nn_w}$;  $(ii)$ the constraints must be satisfied robustly for every uncertainty realization; and $(iii)$ the optimization of the policies is performed over the space of causal functions $\mathcal{C}$. %Indeed, it has been shown that addressing each one of the three issues separately is NP-hard in general. 
The main goal of this paper is to present approximations of instances of \eqref{eq:OptProblem_detGeneral} that can be solved efficiently using tools from convex optimization. In Section~\ref{sec:modUncerSet}, we focus on the optimization of the uncertainty sets without the presence of the control inputs, and discuss several families of uncertainty sets that can be handled efficiently. We then return to the original problem in Section~\ref{sec:decisionRuleApprox}  and show that by applying  affine policy approximation, the resulting optimization problem is convex, and has a number of decision variables and constraints that grows polynomially in the problem parameters.

%%%%%%%%%%%%%%%%%%%%%%%%%%%%%%%%%%%%%%%%%%%
% copied from Angelos, and modified
\section{Uncertainty Set Approximation via Primitive Sets}\label{sec:modUncerSet}
We consider initially the following simplified problem 
\begin{equation}\label{eq:abstract_P}
\begin{array}{ll}
\text{max}  &   \bm\varrho(\mathcal{W})\\[0.3ex]
\text{s.t.}   &   \mathcal{W} \in \mathcal P(\R^{Nn_w}),   \\[0.3ex]
&  \mathbf{D}\mathbf{w} \leq \mathbf{d},\quad \forall\mathbf{w}\in\mathcal{W},
\end{array}
\end{equation}
where $\mathcal{W}=\mathbb{W}_0\times\ldots\times\mathbb{W}_{N-1}$ is the decision variable. 
Problem~\eqref{eq:abstract_P} can be used, for example, to perform robustness analysis on a closed-loop or autonomous systems, where the control policy $\bm\pi(\cdot)$ in \eqref{eq:OptProblem_detGeneral} is fixed. The objective would be to compute the largest set $\mathcal{W}$, according to metric $\bm\varrho(\cdot)$, which the system can tolerate without violating its constraints. Note that Problem~\eqref{eq:abstract_P} becomes intractable when
$\mathcal{W}$ is a polytope with a fixed number of vertices and $\bm\varrho(\cdot)$ is the volume \cite{dyer1988complexity}. 
To gain computational tractability, we restrict $\mathcal{W}$ to be the affine transformation of a \emph{primitive set}, and choose $\bm\varrho(\cdot)$ from a pre-specified family of concave functions.

\begin{definition}[Primitive Set]
Given a convex cone $\mathcal{K}\subseteq\R^{l}$, we call a compact set $\mathbb{S}\subset\R^{n_s}$, $n_s\geq n_w$, a {primitive set} if it has a non-empty relative interior and can be represented as
\begin{equation}\label{eq:simple Set}
\mathbb{S} := \{s\in\R^{n_s}: Gs \preceq_\mathcal{K} g \},
\end{equation}
for some $G\in\R^{l\times n_s}$ and $g\in\R^{l}$. \qed
\end{definition}
\blue{Note that \emph{any} compact convex set is a primitive set, since every convex set admits a conic representation of the form \eqref{eq:simple Set} \cite[p.15]{rockafellar}}. From now on, we \emph{restrict} the uncertainty sets $\mathbb{W}_k$ to be of the form 
\begin{align*}%\label{eq:TrafoS2W_individual}
                & \mathbb{W}_k = Y_k \mathbb{S}_k + y_k := \{Y_k s + y_k: s\in\mathbb{S}_k\} \in \mathcal P(\R^{n_w}), \\
                & \hspace{0.9cm} (Y_k,y_k)\in\mathbb{Y}_k \subseteq (\R^{n_w\times n_s}\times\R^{n_w}),
\end{align*}
where $\mathbb{S}_k$ is a primitive set. \blue{ For simplicity},
%\blue{\sout{Without loss of generality}}, 
we assume that for all stages $k=0,\ldots,N-1$, $\mathbb{S}_k$ is described by the same cone $\mathcal{K}\subseteq\R^l$ and the same matrices $G$ and $g$. \blue{However, all subsequent results can be easily extended to cases when the primitive sets are described by time-varying cones $\mathcal K_k\subset\R^{l_k}$ and matrices $G_k$ and $g_k$.} We refer to $(Y_k,y_k)$ as  \textit{shaping parameters} and assume they take values in a convex set $\mathbb{Y}_k$. Note that if $n_w< n_s$, then $\mathbb{W}_k$ can be seen as the projection of $\mathbb{S}_k$ onto $\R^{n_w}$. Therefore, intuitively speaking, the matrix $Y_k$ can be used to scale, rotate and project the set  $\mathbb{S}_k$, while the vector $y_k$ can be used to translate it.  By defining $\mathcal{S} := \mathbb{S}_0 \times \ldots \times \mathbb{S}_{N-1}$ and
\begin{align*}
\mathcal{Y} &:= \left\{ \mathbf{(Y,y)}\in\R^{Nn_w\times Nn_s}\times\R^{Nn_w} : \exists \{(Y_k,y_k)\in\mathbb{Y}_k\}_{k=0}^{N-1}, \right. \\
                                & \qquad  \left. \mathbf{Y}=\text{diag}(Y_0,\ldots,Y_{N-1}),\ \mathbf{y} = [y_0,\dots,y_{N-1}]  \right\},
\end{align*}
 we can compactly express the restriction on $\mathcal{W}$ as
\begin{equation}\label{eq:TrafoS2W}
\mathcal{W} = \mathbf Y \mathcal{S} + \mathbf y \in \mathcal P(\R^{Nn_w}),\quad (\mathbf{Y,y})\in\mathcal{Y}.
\end{equation}
%The above restriction allows us to parametrize $\mathcal{W}$ with a \emph{finite} number of decision variables. 
Applying restriction  \eqref{eq:TrafoS2W} to problem~\eqref{eq:abstract_P} yields the following problem with decision variables $(\mathcal{W},\mathbf{Y,y})$: % that directly affect the infinite number of constraints:
\begin{subequations}
\begin{equation}\label{eq:abstract_P2}
\begin{array}{ll}
\text{max }&   \bm\varrho(\mathcal{W})\\[0.3ex]
\text{s.t.} &   \mathcal{W} = \mathbf{Y}\mathcal{S}+\mathbf{y},\ \mathbf{(Y, y)} \in\mathcal{Y},   \\[0.3ex]
&  \mathbf{D}\mathbf{w} \leq \mathbf{d},\quad \forall\mathbf{w}\in\mathcal{W}.
\end{array}
\end{equation}
Using standard duality arguments, e.g.\ \cite[Theorem 3.1]{bental:nemirovski:99}, to reformulate the semi-infinite constraint leads to a bilinear and hence non-convex optimization problem. To circumvent this difficulty, we eliminate the variables $\mathbf w$ and $\mathcal W$ in \eqref{eq:abstract_P2} using \eqref{eq:TrafoS2W}. This leads to the following optimization problem with decision-independent uncertainty set $\mathcal{S}$:
\begin{equation}\label{eq:abstract_P_tract2}
\begin{array}{ll}
\text{max}&   \bm\varrho(\mathbf Y\mathcal{S}+\mathbf y)\\[0.3ex]
\text{s.t.} &   \mathbf{(Y, y)} \in\mathcal{Y},   \\[0.3ex]
&  \mathbf{D}(\mathbf{Y s + y}) \leq \mathbf{d},\quad \forall \mathbf{s}\in\mathcal{S},
\end{array}
\end{equation}
\end{subequations}
with decision variables $(\mathbf{Y,y})$. The following proposition shows that problems~\eqref{eq:abstract_P2} and \eqref{eq:abstract_P_tract2} are in fact equivalent.
\begin{proposition}\label{prop:equivalence_simpleProblem}
%Problems~\eqref{eq:abstract_P2} and \eqref{eq:abstract_P_tract2} are equivalent in the following sense: both problems have the same feasible set and the same optimal value.
Problems~\eqref{eq:abstract_P2} and \eqref{eq:abstract_P_tract2} are equivalent in the following sense: both problems have the same optimal value, and there exists a mapping between feasible solutions in both problems.
\end{proposition}
\begin{proof}
The equivalence follows immediately from statement $(i)$ in Lemma \ref{lem:liftingOperators}  in the Appendix, and by substituting $\mathcal{W}$ with $\mathbf{Y}\mathcal{S}+\mathbf{y}$.
\end{proof}
%Proposition \ref{prop:equivalence_simpleProblem} allows us to  solve problem \eqref{eq:abstract_P_tract2} instead of \eqref{eq:abstract_P2}, and resort to techniques 
Since $\mathcal S$ is convex, the semi-infinite constraint in \eqref{eq:abstract_P_tract2} can be reformulated using techniques based on strong duality of convex optimization, see e.g.\ \cite[Theorem 3.1]{bental:nemirovski:99}. This leads to the following finite-dimensional optimization problem:
\begin{equation}\label{eq:abstract_P_tract_ref}
\begin{array}{ll}
\textnormal{max}   & \bm \varrho(\mathbf{Y}\mathcal{S}+\mathbf{y}) \\[0.3ex]
\textnormal{s.t.}   & \mathbf{(Y,y)}\in\mathcal{Y},\, \bm\Lambda\in\R^{N(n_f+n_g) \times Nl},\, \bm\Lambda\succeq_{\bar{\mathcal K}^\star} 0,\\[0.3ex]
& \mathbf{D y} + \bm\Lambda  \mathbf{g} \leq \mathbf{d},\, \bm\Lambda \mathbf{G} = \mathbf{DY}, 
\end{array}
\end{equation}
with decision variables $(\mathbf{Y,y},\bm{\Lambda})$, where $\mathbf{G} := \text{diag}(G,\ldots,G)\in\R^{Nl\times Nn_s}$ $\mathbf{g} := [g,\ldots,g]\in\R^{Nl}$, $\bar{\mathcal{K}}:=\mathcal{K} \times \ldots \times \mathcal{K}\subset\R^{Nl}$ and $\bar{\mathcal{K}}^\star$ is its dual cone. Problem~\eqref{eq:abstract_P_tract_ref} has convex constraints, and constitutes an inner approximation of the infinite-dimensional problem \eqref{eq:abstract_P} due to restriction \eqref{eq:TrafoS2W}. Note that both the number of decision variables and the number of constraints in problem~\eqref{eq:abstract_P_tract_ref} grow polynomially with respect to the problem parameters $(N,n_x,n_s,n_w,n_f,n_g,l)$. If in addition the objective function $\bm\varrho(\cdot)$ is concave in the shaping variables $(\mathbf Y, \mathbf y)$, then problem~\eqref{eq:abstract_P_tract_ref} is a tractable convex optimization problem.

In the following, we discuss choices of the function $\bm\varrho(\cdot)$ that will render problem \eqref{eq:abstract_P_tract_ref} convex, and  show that a large family of uncertainty sets of practical interest is captured by the affine transformation \eqref{eq:TrafoS2W}.

\subsection{Volume Maximization}\label{sec:volMaximization}
A natural objective in problem \eqref{eq:abstract_P} is to maximize the volume of the sets $\mathbb{W}_k$. Indeed, for the special case of $n_s=n_w$, it is well-known that $\text{vol}(\mathbb{W}_k)=\sqrt{\text{det}(Y_k^\top Y_k)}\,\text{vol}(\mathbb{S}_k)$ \cite[Section 8.3.1]{boyd2004convex}, where $\text{vol}(\mathbb{W}_k)$ and $\text{vol}(\mathbb{S}_k)$ are the Lebesgue measures of the sets $\mathbb{W}_k$ and $\mathbb{S}_k$, respectively. If $\mathbb{Y}_k$ is chosen such that $Y_k$ is constrained to be positive-definite, i.e.\ $\mathbb{Y}_k:=\{(Y_k,y_k)\in(\R^{n_w \times n_w}\times \R^{n_w}):Y_k\succ0\}$, then maximizing $\varrho(Y_k\mathbb{S}_k+y_k)=\log\det(Y_k)$ maximizes $\text{vol}(\mathbb{W}_k)$, see \cite[Section 3.1.5]{boyd2004convex} for more details\footnote{\blue{Similarly, the volume of $\mathcal W$ is maximized by choosing $\bm\varrho(\mathbf Y\mathcal S+\mathbf y)=\log\det(\mathbf Y)$, subject to $n_w=n_s$ and $Y_k\succ0$.}}. Since $\log\det(\cdot)$ is a concave function, problem \eqref{eq:abstract_P_tract_ref} is a convex optimization problem that can be solved efficiently \cite{VandenbergheDetMax1998}. \blue{For the case $n_s>n_w$, maximizing the volume of $\mathbb{W}_k$ is generally a non-convex optimization problem, and a convex, albeit heuristic, objective function should be chosen to keep the optimization problem tractable. Examples of such cost functions are given in Section~\ref{sec:familiesUncertaintySets} below. }

\subsection{Specific Families of Uncertainty Sets}\label{sec:familiesUncertaintySets}
In practical applications, uncertainty sets $\mathbb{W}_k$ are typically constrained to have a specific geometric form, such as ellipsoidal or rectangular. This requirement can be easily incorporated in problem \eqref{eq:abstract_P} by including an additional constraint of the form $\mathbb{W}_k\in\Omega_k\subset\mathcal{P}(\R^{n_w})$, where $\Omega_k$ is the family of admissible uncertainty sets. In the following, we study families of uncertainty sets that often arise in robust control problems, focusing on norm balls, ellipsoids, hyper-rectangles and polytopes. We show that these representations can be obtained by appropriately selecting the primitive set $\mathbb{S}_k$ and constraint set $\mathbb{Y}_k$. To simplify notation, we omit the time indices $k$ in the subsequent discussion.

\subsubsection*{Ball Uncertainty Set}
$p$-norm ball uncertainty sets, with $p\in[1,\infty]$, take the form $\mathbb{W}=\{w\in\R^{n_w}: \|w-y\|_p \leq r\}$, with parameters \mbox{$r\geq0$}  and $y\in\R^{n_w}$. In the spirit of \cite{zhangCDC2014}, we choose the primitive set $\mathbb{S}=\{s\in\R^{n_w}: \|s\|_p \leq 1\}$, where $\mathcal K$ in \eqref{eq:simple Set} is the $p$-order cone, and restrict $\mathbb{Y}=\{(Y,y)\in(\R^{n_w \times n_w}\times \R^{n_w}):\exists r\in\R_+,\, Y=rI\}$, where $I\in\R^{n_w\times n_w}$ is the identity matrix. This allows us to express $p$-norm ball uncertainty sets as
 $\mathbb{W}=Y\mathbb{S}+y$, subject to $(Y,y)\in\mathbb{Y}$. A natural choice for the objective function is $\varrho(\mathbb{W})=r$, which is proportional to $\log\det(Y)=n_w r$ and thus maximizes the volume of $\mathbb{W}$. Since $\mathcal{K}^\star$ is the $q$-order cone, where $1/p+1/q=1$, and the set $\mathbb{Y}$ is described by linear constraints, problem \eqref{eq:abstract_P_tract_ref} is a conic convex optimization problem that can be solved efficiently \cite{XueYe2000efficient}.

\subsubsection*{Ellipsoidal Uncertainty Set}
Ellipsoidal uncertainty sets take the form $\mathbb{W}=\{w\in\R^{n_w}: (w-y)^\top \Sigma^{-1} (w-y)\leq1 \}$, with parameters $\Sigma=\Sigma^\top\succ0$ and $y\in\R^{n_w}$. By choosing the primitive set $\mathbb{S}=\{s\in\R^{n_w}:\|s\|_2\leq1\}$, where $\mathcal K$ in \eqref{eq:simple Set} is the second-order cone, and setting $\mathbb{Y}=\{(Y,y)\in(\R^{n_w \times n_w}\times \R^{n_w}): Y=Y^\top \succ 0\}$, we can express the ellipsoid as  $\mathbb{W}=Y\mathbb{S}+y$ by identifying $Y$ with $\Sigma^{1/2}$, see e.g.\  \cite[Section 2.2.2]{boyd2004convex}. The volume of $\mathbb{W}$ is maximized by choosing $\varrho(Y\mathbb{S}+y)=\log\det(Y)$. Since $\mathcal K^\star=\mathcal K$ and $\mathbb{Y}$ is described by linear matrix inequalities, problem~\eqref{eq:abstract_P_tract_ref} can be solved efficiently \cite{VandenbergheDetMax1998}.

\subsubsection*{Axis-aligned Hyper-Rectangular Uncertainty Sets}
Axis-aligned rectangular uncertainty sets take the general form $\mathbb{W}=\{w\in\R^{n_w}:-\gamma \leq w-y \leq \gamma\}$, with parameters $\gamma\in\R^{n_w}_{+}$ and $y\in\R^{n_w}$. By choosing $\mathbb{S}=\{s\in\R^{n_w}:\|s\|_\infty\leq1\}$, where $\mathcal K$ in \eqref{eq:simple Set} is the $\infty$-order cone, and $\mathbb{Y}=\{(Y,y)\in(\R^{n_w \times n_w}\times \R^{n_w}): \exists\gamma\in\R_+^{n_w},\,Y=\text{diag}(\gamma)\}$, we can write $\mathbb{W}=Y\mathbb{S}+y$ subject to $(Y,y)\in\mathbb{Y}$. The volume of $\mathbb{W}$ is maximized with $\varrho(Y\mathbb{S}+y)=\log\det(Y)=\sum_{i=1}^{n_w}\log Y_{ii}$. For given $c_i\in\R$, $i\in\{1,\ldots,n_w\}$, another choice of the objective could be $\varrho(Y\mathbb{S}+y)=\sum_{i=1}^{n_w} c_i Y_{ii}$, which maximizes the (weighted) circumference of the rectangle $\mathbb{W}$. Since the dual cone $\mathcal{K}^\star$ is the first-order cone and $\mathbb{Y}$ is described by linear constraints, problem~\eqref{eq:abstract_P_tract_ref} is a convex optimization problem for either choice of the objective function.

\subsubsection*{Polyhedral Uncertainty Set}
Compact polytopes with $m\geq n_w$ vertices take the general form $\mathbb{W} = \text{conv}(v^{(1)},\ldots,v^{(m)})$, with parameters $(v^{(1)},\ldots,v^{(m)})$. By choosing the primitive set as the simplex $\mathbb{S}=\text{conv}(e_1,\ldots,e_{m})$, where $e_i\in\R^{m}$ is the $i$-th unit vector and $\mathcal K=\R^m_+$ in \eqref{eq:simple Set}, and $\mathbb{Y}=\{(Y,y)\in(\R^{n_w \times m}\times \R^{n_w}): y= 0\}$, the polytope can be expressed as $\mathbb{W}=Y\mathbb{S}$, where $Y_{\cdot j}$ corresponds to $v^{(j)}$ for all $j=1,\ldots,m$. Unfortunately, maximizing the volume of a generic polytope with $m$ vertices is computationally intractable \cite{dyer1988complexity}. Instead, one can, \blue{for example,} use the linear objective function $\varrho(Y\mathbb{S})=\sum_{j=1}^{m} c_j^\top Y_{\cdot j}$, which places the individual vertices of $\mathbb{W}$ in user-specified directions $c_j\in\R^{n_w}$, $j=1,\ldots,m$ (``pushing"). With all constraints linear, the resulting problem is a linear optimization problem. \blue{ An alternative cost function is $\varrho(Y\mathbb{S})=-\sum_{j=1}^{m} \|d_j - Y_j\|_2^2$, which places the vertices of $\mathbb W$ close to chosen points $d_j\in\R^{n_w}$, $j=1,\ldots,m$ (``pulling"),  resulting in a convex quadratic optimization problem.  }

%%%%%%%%%%%%%%%%%%%%%%%%%%%%%%%%%%%%%%%%%%%
\section{Policy Approximation via Affine Decision Rules}\label{sec:decisionRuleApprox}
To provide a tractable approximation for optimizing over policies, we first restrict the  uncertainty set $\mathcal{W}$ to admit the affine representation \eqref{eq:TrafoS2W}, giving rise to the following infinite-dimensional problem:
\begin{subequations}
\begin{equation}\label{eq:OptProblem_detGeneral_1}
\begin{array}{ll}
\text{min}  & \displaystyle \max_{\mathbf{w}\in\mathcal{W}} \left\{\mathbf{c}^\top \bm\pi(\mathbf{w})\right\} - \lambda \bm\varrho(\mathcal{W}) \\[0.3ex]
\text{s.t.} &  \bm\pi(\cdot)\in\mathcal{C},\ \mathcal{W}=\mathbf{Y}\mathcal{S}+\mathbf{y},\ (\mathbf{Y},\mathbf{y})\in\mathcal{Y},\\[0.3ex]
& \mathbf{C}\bm\pi(\mathbf{w}) + \mathbf{D}\mathbf{w} \leq \mathbf{d},\qquad \forall\mathbf{w}\in\mathcal{W}, \end{array}
\end{equation}
with decision variables $(\bm\pi(\cdot),\mathcal{W},\mathbf{Y,y})$. Similar to problem~\eqref{eq:abstract_P2}, problem~\eqref{eq:OptProblem_detGeneral_1} has a uncertainty set parametrized  by the decision variables $(\mathbf{Y,y})$.
Therefore, we consider the following reformulation of problem~\eqref{eq:OptProblem_detGeneral_1}, parametrized  in terms of $\mathbf{s}$, and with decision-independent uncertainty set:
\begin{equation}\label{eq:OptProblem_detGeneral_2}
\begin{array}{ll}
\text{min} &  \displaystyle\max_{\mathbf{s}\in\mathcal{S}} \left\{\mathbf{c}^\top \bm{\widetilde\pi}(\mathbf{s})\right\} - 
\lambda\bm\varrho(\mathbf Y\mathcal{S}+\mathbf y)\\ [0.3ex]
\text{s.t.} &  \bm{\widetilde\pi}(\cdot)\in\widetilde{\mathcal{C}},\ (\mathbf{Y},\mathbf{y})\in\mathcal{Y}, \\[0.3ex]
& \mathbf{C}\bm{\widetilde\pi}(\mathbf{s}) + \mathbf{D}(\mathbf{Y}\mathbf{s}+\mathbf{y}) \leq \mathbf{d}, \quad \forall\mathbf{s}\in\mathcal{S}, 
\end{array}
\end{equation}
\end{subequations}
with decision variables $(\bm{\widetilde\pi}(\cdot),\mathbf{Y,y})$, and where $\widetilde{\mathcal{C}} := \{[\widetilde\pi_0(\cdot),\ldots,\widetilde\pi_{N-1}(\cdot)]:\,\widetilde\pi_{k}:\mathbb{S}_0\times\cdots\times\mathbb{S}_k\rightarrow\R^{n_u},\,k=0,\ldots,N-1\}$. The set $\widetilde{\mathcal{C}}$ ensures that  $\bm{\widetilde\pi}(\cdot)$ are causal policies, depending only on the first $k$ elements of $\mathbf{s}$. The following proposition shows that problems~\eqref{eq:OptProblem_detGeneral_1} and \eqref{eq:OptProblem_detGeneral_2} are equivalent.

\begin{proposition}\label{prop:equiv_reparaPolicies}
Problems~\eqref{eq:OptProblem_detGeneral_1} and \eqref{eq:OptProblem_detGeneral_2} are equivalent in the following sense: both problems have the same optimal value, and there exists a (not necessar\blue{ily} unique) mapping between feasible solutions in both problems.
\end{proposition}
\begin{proof}
The proof is provided in the Appendix.
\end{proof}

Proposition~\ref{prop:equiv_reparaPolicies} allows us to focus on problem~\eqref{eq:OptProblem_detGeneral_2}, which belongs to the class of multistage adaptive optimization problems, involving a continuum of decision variables and inequality constraints \cite{BertsimasGoyalPowerLimitations2012}. %\cite{guslitser2002uncertainty, BertsimasGoyalPowerLimitations2012}. 
 Following \cite{ben2004adjustable}, we restrict the space of admissible policies to exhibit the following affine structure 
\begin{equation}\label{eq:det_ADF}
\displaystyle\widetilde\pi_k(s_0,\ldots,s_k) := p_k + \sum_{j=0}^k P_{k,j} s_j,\quad k=0,\ldots,N-1,
\end{equation}
for some $P_{k,j}\in\mathbb{R}^{n_u \times n_s}$ and $p_k\in\mathbb{R}^{n_u}$. The concatenated policy is expressed by $\bm{\widetilde\pi}(\mathbf{s}) = \mathbf{P}\mathbf{s}+\mathbf{p}$, where $\mathbf{P}\in\R^{Nn_u\times N n_s}$ is a lower-block-triangular matrix defined by the elements $P_{k,j}$ and $\mathbf{p}:=[p_0,\ldots,p_{N-1}]\in\R^{Nn_u}$. This restriction yields the following   semi-infinite problem,
\begin{equation}\label{eq:det_OptProblem}
\begin{array}{ll}
\text{min}  & \displaystyle \max_{\mathbf{s}\in\mathcal{S}} \left\{\mathbf{c}^\top (\mathbf{P}\mathbf{s}+\mathbf{p}) \right\} - \lambda\bm\varrho(\mathbf Y\mathcal{S}+\mathbf y)\\[0.3ex]
\text{s.t.}    &  \mathbf{P}\in\R^{Nn_u\times N n_s},\ \mathbf{p}\in\R^{Nn_u},\ (\mathbf{Y},\mathbf{y})\in\mathcal{Y}, \\[0.3ex]
& \mathbf{C}(\mathbf{P}\mathbf{s}+\mathbf{p}) + \mathbf{D}(\mathbf{Y}\mathbf{s}+\mathbf{y}) \leq \mathbf{d}, \quad \forall\mathbf{s}\in\mathcal{S},
\end{array}
\end{equation}
with a finite number of decision variables $(\mathbf{P,p,Y,y})$ and an infinite number of constraints. Since $\mathcal S$ is convex, one can employ the duality argument of \cite[Theorem 3.1]{bental:nemirovski:99} to reformulate problem~\eqref{eq:det_OptProblem} into the following finite-dimensional convex optimization problem:
\begin{equation}\label{eq:det_OptProblem_reformulated}
\begin{array}{ll}
\text{min}\quad & \tau - \lambda\bm\varrho(\mathbf Y\mathcal{S}+\mathbf y) \\[0.5ex]
\text{s.t.} \quad\  &  \tau\in\R,\ \mathbf{P}\in\R^{Nn_u\times N n_s},\ \mathbf{p}\in\R^{Nn_u},\ (\mathbf{Y},\mathbf{y})\in\mathcal{Y},\\[0.3ex]
& \bm{\mu}\in\R^{Nl},\,\bm{\mu}\succeq_{\bar{\mathcal K}^\star} 0,\,\bm\Lambda\in\R^{N(n_f + n_g) \times Nl},\, \bm\Lambda\succeq_{\bar{\mathcal K}^\star} 0,  \\ [0.3ex]
&\bm{c}^\top\bm{p} + \bm{\mu}^\top\bm{g}\leq \tau,\, \mathbf{G}^\top\bm{\mu} = \mathbf{P}^\top\bm{c},\\[0.3ex]
&\mathbf{C p} + \mathbf{D y} + \bm\Lambda  \mathbf{g} \leq \mathbf{d},\, \bm\Lambda \mathbf{G} = \mathbf{CP} + \mathbf{DY},
\end{array}
\end{equation}
with decision variables $(\tau,\mathbf{P,p,Y,y},\bm{\mu,\Lambda})$. Note that the size of problem~\eqref{eq:det_OptProblem_reformulated} grows polynomially in the parameters $(N,n_x,n_u,n_s,n_w,n_f,n_g,l)$. Furthermore, if $\bm\varrho(\cdot)$ is chosen as discussed in Sections~\ref{sec:volMaximization} and \ref{sec:familiesUncertaintySets}, then problem~\eqref{eq:det_OptProblem_reformulated} is a convex optimization problem that can be solved efficiently;
%The class to which problem~\eqref{eq:det_OptProblem_reformulated} belongs to (i.e.\ linear program, second order cone program, semi-definite program), depends only on the cone $\bar{\mathcal{K}}^\star$ and the structure of $\mathcal{Y}$. 
in particular, the affine policy does not alter the problem class.% Therefore, the discussion in Section~\ref{sec:familiesUncertaintySets} directly carries over to problem~\eqref{eq:det_OptProblem_reformulated}. 

The following theorem summarizes the relationship between the original problem \eqref{eq:OptProblem_detGeneral}, and its approximation \eqref{eq:det_OptProblem_reformulated}, obtained through the restrictions \eqref{eq:TrafoS2W} and \eqref{eq:det_ADF}:
\begin{theorem}\label{thm:innerApprox}
Problem~\eqref{eq:det_OptProblem_reformulated} constitutes an inner approximation of the original problem~\eqref{eq:OptProblem_detGeneral} in the following sense: every feasible solution of \eqref{eq:det_OptProblem_reformulated} is feasible in \eqref{eq:OptProblem_detGeneral}, and an upper bound to problem~\eqref{eq:OptProblem_detGeneral} can be found by solving problem~\eqref{eq:det_OptProblem_reformulated}.
\end{theorem}
%\begin{proof}
%        \red{Due to strong duality of convex programming, problem \eqref{eq:det_OptProblem_reformulated} is equivalent to problem \eqref{eq:det_OptProblem}, which itself constitutes an inner approximation of problem \eqref{eq:OptProblem_detGeneral_2} due to restriction \eqref{eq:det_ADF}. The statement follows immediately since problems \eqref{eq:OptProblem_detGeneral_2} and \eqref{eq:OptProblem_detGeneral_1} are equivalent, see Proposition \eqref{prop:equiv_reparaPolicies}, and because problem \eqref{eq:OptProblem_detGeneral_1} inner approximates the original problem \eqref{eq:OptProblem_detGeneral} due to restriction \eqref{eq:TrafoS2W}.}
%\end{proof}

%%%%%%%%%%%%%%%%%%%%%%%%%%%%%%%%%%%%%%%%%%

\subsection{Relationship between the policies \boldmath{$\pi(\cdot)$} and \boldmath{$\widetilde\pi(\cdot)$} }\label{subsec:RelationshipPolicies}

By restricting our attention to affine policies $\bm{\tilde{\pi}}(\cdot)$ in \eqref{eq:det_ADF}, we were able to obtain a tractable optimization problem in \eqref{eq:det_OptProblem_reformulated}. By Proposition \ref{prop:equiv_reparaPolicies}, these affine policies  are equivalent to some, possibly non-linear, policy $\bm\pi(\cdot)\in\mathcal C$. In the following, we show that if $\mathcal S$ is a polytope, then affine policies $\widetilde{\bm\pi}(\cdot)$ correspond to continuous piece-wise affine policies $\bm\pi(\cdot)$.
To see this, let $\mathcal{\widetilde C}_\text{aff}\subset\mathcal{\widetilde C}$ be the class of causal affine policies in problem~\eqref{eq:OptProblem_detGeneral_2} defined as in  \eqref{eq:det_ADF}; $\mathcal{C}_\textnormal{aff}\subset\mathcal{C}$ be the class of causal affine policies in problem~\eqref{eq:OptProblem_detGeneral_1}; and $\mathcal{C}_\textnormal{pwa}\subset\mathcal{C}$ be the class of causal, piece-wise affine, continuous policies in problem~\eqref{eq:OptProblem_detGeneral_1}. 
\begin{proposition}\label{prop}
Let $\mathcal{S}:=\mathbb S_0\times\ldots\times\mathbb S_{N-1}$. Then, the following hold:
\begin{itemize}\label{prop:relationshipPolicyW_and_PolicyS}
\item[(i)] For every  $(\bm\pi(\cdot),\mathcal W, (\mathbf Y, \mathbf y))\in\mathcal{C}_\textnormal{aff} \times \mathcal P(\R^{Nn_w}) \times \mathcal Y$ feasible in problem~\eqref{eq:OptProblem_detGeneral_1}, there exists a $(\bm{\widetilde\pi}(\cdot),(\mathbf Y,\mathbf y))\in\mathcal{\widetilde C}_\textnormal{aff}\times\mathcal Y$ feasible in problem~\eqref{eq:OptProblem_detGeneral_2} that achieves the same objective value.
\item[(ii)] \blue{If $\mathbf Y$ is invertible, then for every $(\bm{\widetilde\pi}(\cdot),(\mathbf Y,\mathbf y))\in\tilde{\mathcal C}_\textnormal{aff}\times\mathcal Y$ feasible in problem~\eqref{eq:OptProblem_detGeneral_2}, there exists a $(\bm\pi(\cdot),\mathcal W,(\mathbf Y,\mathbf y))\in\mathcal C_\textnormal{aff} \times \mathcal P(\R^{Nn_w}) \times \mathcal Y$ feasible in problem~\eqref{eq:OptProblem_detGeneral_1} that achieves the same objective. }
\item[(iii)] If $\mathcal{S}$ is a polytope, then for every $(\bm{\widetilde\pi}(\cdot),(\mathbf Y,\mathbf y))\in\mathcal{\widetilde C}_\textnormal{aff} \times \mathcal Y$ feasible in problem~\eqref{eq:OptProblem_detGeneral_2}, there exists a $(\bm\pi(\cdot),\mathcal W,(\mathbf Y,\mathbf y))\in\mathcal{C}_\textnormal{pwa}\times \mathcal P(\R^{Nn_w}) \times \mathcal Y$ feasible in problem~\eqref{eq:OptProblem_detGeneral_1} that achieves the same objective value.
\end{itemize}
\end{proposition}
\begin{proof}
The proof is provided in the Appendix.
\end{proof}
From Proposition~\ref{prop} $(i)$ it follows that applying affine policies to problem \eqref{eq:OptProblem_detGeneral_2} produces at least as good results as applying affine policies to \eqref{eq:OptProblem_detGeneral_1}.  \blue{From part $(ii)$ it follows that if $\mathbf Y$ is invertible (e.g.\ $\mathbf Y\succ0$), then imposing $\bm{\widetilde\pi}(\cdot)\in\widetilde{\mathcal C}_\text{aff}$ in problem~\eqref{eq:OptProblem_detGeneral_2} is equivalent to imposing $\bm\pi\in\mathcal C_\text{aff}$ in problem~\eqref{eq:OptProblem_detGeneral_1}, but has the advantage that a convex optimization problem is solved instead of a bilinear optimization problem.} %Finally, it follows from part $(iii)$ of Proposition~\ref{prop} that if $\mathbf Y$ is not invertible, then imposing $\bm{\widetilde\pi}(\cdot)\in\widetilde{\mathcal C}_\text{aff}$  will result $\bm{\pi}(\cdot)\in{\mathcal C}_\text{pwa}$, which is indeed more general than imposing $\bm{\pi}(\cdot)\in{\mathcal C}_\text{aff}$.} 
\blue{ Finally, it follows from the proof of Proposition~\ref{prop} that if $(\mathbf{P}^\star,\mathbf{p}^\star,\mathbf{Y^\star,y^\star})$ are the optimizers of problem~\eqref{eq:det_OptProblem_reformulated}, then a causal policy $\bm\pi(\cdot)\in\mathcal C$ can be obtained as
\begin{equation}\label{eq:recoveringW}
	\bm\pi(\mathbf{w}) := \mathbf{P}^\star \left[ (\mathbf{Y}^\star)^{-1}(\mathbf{w} - \mathbf{y}^\star) \right] + \mathbf{p}^\star
\end{equation}
if $\mathbf Y^\star$ is invertible, and otherwise as
\begin{equation}\label{eq:recoveringW_pwa}
\bm\pi(\mathbf w):=\mathbf{P}^\star L(\mathbf{w}) + \mathbf{p}^\star, 
\end{equation}
where $L:\mathcal{W}\to\mathcal{S}$ is a so-called \emph{lifting operator} (see Appendix). From a practical point of view, \eqref{eq:recoveringW} and \eqref{eq:recoveringW_pwa} also provide a way for computing the control input if a disturbance realization $\mathbf{w}\in\mathcal{W}$ is given or measured. }

\blue{We close this section by pointing out that instead of approximating $\bm{\tilde\pi}(\cdot)$ with an affine-policy, it can be approximated using more complex, e.g.\ piece-wise affine, policies. Similar to the case of $\bm{\tilde\pi}(\cdot)\in\tilde{\mathcal C}_\text{aff}$, it can be shown that if $\bm{\tilde\pi}(\cdot)\in\tilde{\mathcal C}_\text{pwa}$ and $\mathcal{S}$ is a polytope, then $\bm\pi(\cdot)$ will be piece-wise affine. % and at least as good as if $\bm{\tilde\pi}(\cdot)$ was restricted to be affine. 
The interested reader is referred to  \cite{georghiouGeneralized2015, GeorghiouBinaryMP2015} for tractable methods of designing piece-wise affine policies $\bm{\tilde\pi}(\cdot)$ using convex optimization techniques. }

\section{\blue{Example 1:} Frequency Reserve Provision}\label{sec:example}
%In this section, we study a reserve provision problem that arises in power systems. We show how such a problem can formulated as a robust optimal control problem with adjustable uncertainty sets, which we approximate and solve using the methodology developed in Sections \ref{sec:modUncerSet} and \ref{sec:decisionRuleApprox}. It is well-known that i
In a power grid, the supply and demand of electricity must be balanced at all times to maintain frequency stability. This balance is achieved by the grid operator by procuring so-called frequency reserves that are dispatched whenever supply does not meet demand \cite{rebours2007survey}. Recently, there has been increasing interest in so-called demand response schemes, where flexible loads such as office buildings are used to provide such reserves \cite{callaway2011achieving}. Buildings participating in the demand response program are asked in real-time by the grid operator to adjust their electricity consumption within a given range (the so-called \emph{reserve capacity}). They are rewarded financially for this service through a payment that typically depends both on the size of the offered reserve capacity and on the actual reserves provided. In the following, we investigate the amount of frequency reserves commercial buildings can offer to the power grid.
%In the following, we show that the question of providing the optimal (frequency) reserve capacity, which we will denote as $\mathbb{W}_k$ at time $k$, can be formulated as a robust control problem with adjustable uncertainty sets, that can be solved using the  tools developed in the previous sections.

\subsection{Problem Formulation}
Following \cite{Oldewurtel_PhD2011,Oldewurtel:energyBuildings:2012}, we assume that the building dynamics are of the form $x_{k+1} = Ax_k + B(u_k + \Delta u_k) + Ev_k$, where the states represent the temperatures in the room, walls, floor, and ceiling. The affine term $v_k\in\R^{n_v}$ represents uncontrolled inputs acting on the building, such as ambient temperature, solar radiation, and internal gains (number of building occupants, etc.). For simplicity, we assume that $v_k$ can be predicted perfectly, and refer to \cite{Oldewurtel_PhD2011, zhang:schildbach:sturzenegger:morari:13} for methods of estimating these. The control inputs consist of the building's heating, ventilation and air condition (HVAC) system, which we model through two input vectors $u_k$ and $\Delta u_k$. We interpret the (strictly causal) input $u_k$ as the nominal control input which we apply if no reserve is demanded at time $k$. The (causal) input $\Delta u_k$ models the necessary input change at time $k$ to follow the actual \emph{reserve demand} $w_k\in\mathbb{W}_k$, where  $\mathbb{W}_k\subseteq\R$ is the reserve capacity.  Note that $\Delta u_k$ is required to depend on $w_k$ since it is used in real-time to follow the reserve demand $w_k$, which is a regulation signal sent  by the grid operator.
Following the literature \cite{Oldewurtel:energyBuildings:2012, zhang:schildbach:sturzenegger:morari:13}, we model the electricity consumption of the HVAC system through a linear function of the form $u_k \mapsto \eta^\top u_k$, so that $e_k := \eta^\top u_k$ is the nominal electricity consumption, and $\Delta e_k := \eta^\top \Delta u_k$ is the change in electricity consumption if the input is changed from $u_k$ to $u_k+\Delta u_k$. In practice, $\Delta u_k$ is chosen such that it matches the reserve demand, satisfying $w_k = \eta^\top \Delta u_k$ for all $k=0,\ldots,N-1$.

The objective function is composed of two competing quatities: the maximization of reserve capacity, and the minimization of electricity cost. The electricity cost takes the form $\sum_{k=0}^{N-1} c_k\,  \eta^\top u_k$, where $c_k\in\R_+$ is the electricity price at time $k$. Following the current norm in Switzerland \cite{SwissGridBasicPrinciplesAS_2015}, 
%\footnote{\scriptsize{\url{https://www.swissgrid.ch/dam/swissgrid/experts/ancillary_services/Dokumente/D151019_AS-Products_V9R1_en.pdf}}}, 
we consider symmetric reserve capacity $\mathbb{W}_k = [-Y_k, Y_k]$, where $Y_k\in\R_+$. We assume that offering reserve capacity $\mathbb{W}_k$ of size $Y_k$ is compensated with a  reward of the form $\lambda\, Y_k$. For simplicity, the reward associated with the actual reserves provided $\mathbf{w} := [w_0,\ldots,w_{N-1}]$ is not considered.  
%The objective is to compute the optimal size of the (symmetric) reserve capacity set, which we parametrize through its bounds $Y_k\in\R_+$ as $\mathbb{W}_k(Y_k) = [-Y_k, Y_k]$. In practice, $Y_k$ and $-Y_k$ correspond the the maximal and minimal reserve the building can provide at time $k$. We assume that offering reserve capacity $\mathbb{W}_k(Y_k)$ is compensated with a linear reward of the form $\lambda\, Y_k$. The overall objective is to minimize the nominal electricity cost $J(\mathbf{u}) = \sum_{k=0}^{N-1} c_k\,  \xi^\top u_k$, where $c_k\in\R_+$ is the electricity price at time $k$, while maximizing the total reward $\lambda\sum_{k=0}^{N-1} Y_k$. 
Since the reserve demand $\mathbf{w}$ is uncertain at the time of planning, we introduce the policies $\bm{\pi}(\cdot):=[\pi_0(\cdot),\ldots,\pi_{N-1}(\cdot)]\in\mathcal{C}_s$ and $\bm{\kappa}(\cdot):=[\kappa_0(\cdot),\ldots,\kappa_{N-1}(\cdot)]\in\mathcal{C}$ for $[u_0,\ldots,u_{N-1}]$ and $[\Delta u_0,\ldots,\Delta u_{N-1}]$, respectively, where $\mathcal{C}_s$ is the class of all strictly causal policies. 
\blue{Note that the inputs $u_k=\pi_k(w_0,\ldots,w_{k-1})$ are strictly causal because in practice they are decided and applied independent of the current realization $w_k$. In contrast, the inputs $\Delta u_k=\kappa_k(w_0,\ldots,w_{k})$ are only causal  since they must satisfy the (instantaneous) reserve request $w_k=\eta^\top \Delta u_k$ for all $w_k\in\mathbb W_k$. }

\blue{The frequency reserve provision problem can now be formulated as the following robust optimal control problem with adjustable uncertainty set:}
%This leads to the following robust optimal control problem with adjustable uncertainty set:
\begin{equation}\label{eq:example_demandResponse}
\begin{array}{ll}
\hspace{-0.1cm} 	\text{min}  & \displaystyle\max_{\mathbf{w}\in\mathcal{W}} \left\{\mathbf{c}^\top \mathbf{e}(\mathbf{w})\right\} - \lambda\, \bm\varrho(\mathcal{W}) \\[0.3ex]
\hspace{-0.05cm} 	\text{s.t.}     	&  \bm\pi(\cdot)\in\mathcal{C}_s,\ \bm{\kappa}(\cdot)\in\mathcal{C},\  \mathcal{W}=\times_{k=0}^{N-1}[-Y_k,Y_k],\ \mathbf{Y}\geq0  \\[0.3ex]
                        	& \!\!\!\left. 
                                \begin{array}{l}        \mathbf{x}(\mathbf{w}) = \mathbf{A} x_0 + \mathbf{B}(\bm\pi(\mathbf{w}) + \bm\kappa(\mathbf{w}))+\mathbf{E}\mathbf{v},\\[0.3ex]
                                                                \mathbf{F}\,\mathbf{x}(\mathbf{w}) \leq \mathbf{f},\ \mathbf{G}(\bm\pi(\mathbf{w}) + \bm\kappa(\mathbf{w})) \leq \mathbf{g}, \\[0.3ex]
                                                                \mathbf{e}(\mathbf{w}) = \bm\eta^\top\bm\pi(\mathbf{w}),\  \mathbf{w}=\bm{\eta}^\top\bm{\kappa}(\mathbf{w}),
                                        \end{array}   \hspace{-0.0cm} \right\} \forall\mathbf{w}\in\mathcal{W},
\end{array}
\end{equation}
with decision variables $(\bm\pi(\cdot),\bm\kappa(\cdot),\mathcal{W},\mathbf{Y})$, and $\mathbf{c} := [c_0,\ldots,c_{N-1}]$, $\mathbf{Y}$ is the collection of all $\{Y_k\}_{k=0}^{N-1}$, $\bm\varrho(\mathcal{W}) := \sum_{k=0}^{N-1} Y_k$, $\mathbf{v} := [v_0,\ldots,v_{N-1}]$, $x_0\in\R^{n_x}$ is the initial state, and $\bm\eta^\top := \text{diag}(\eta^\top,\ldots,\eta^\top)$. The matrices $\mathbf{(A,B,E,F,f,G,g)}$ can easily be constructed from the problem data. The inequality constraint   $\mathbf{F\,x(w)\leq f}$, ensures that comfort constraints are satisfied at all times. Similarly, the inequality constraint $\mathbf{G}(\bm\pi(\mathbf{w}) + \bm\kappa(\mathbf{w})) \leq \mathbf{g}$, ensures that the constraints on the HVAC systems are always satisfied. The first equality constraint describes the state dynamics, while  constraint $\mathbf{e}(\mathbf{w}) = \bm\eta^\top\bm\pi(\mathbf{w})$ models the (nominal) power consumption. Finally, the constraint $\mathbf{w} = \bm{\eta}^\top\bm{\kappa}(\mathbf{w}) $ ensures that for every reserve demand sequence $\mathbf{w}$, there exists a feasible input change $\bm\kappa(\mathbf{w})$ that matches the reserve demand.

\begin{figure}[t!]
        \centering
        \includegraphics[trim = 11mm 13mm 20mm 12mm, clip, width=0.70\textwidth]{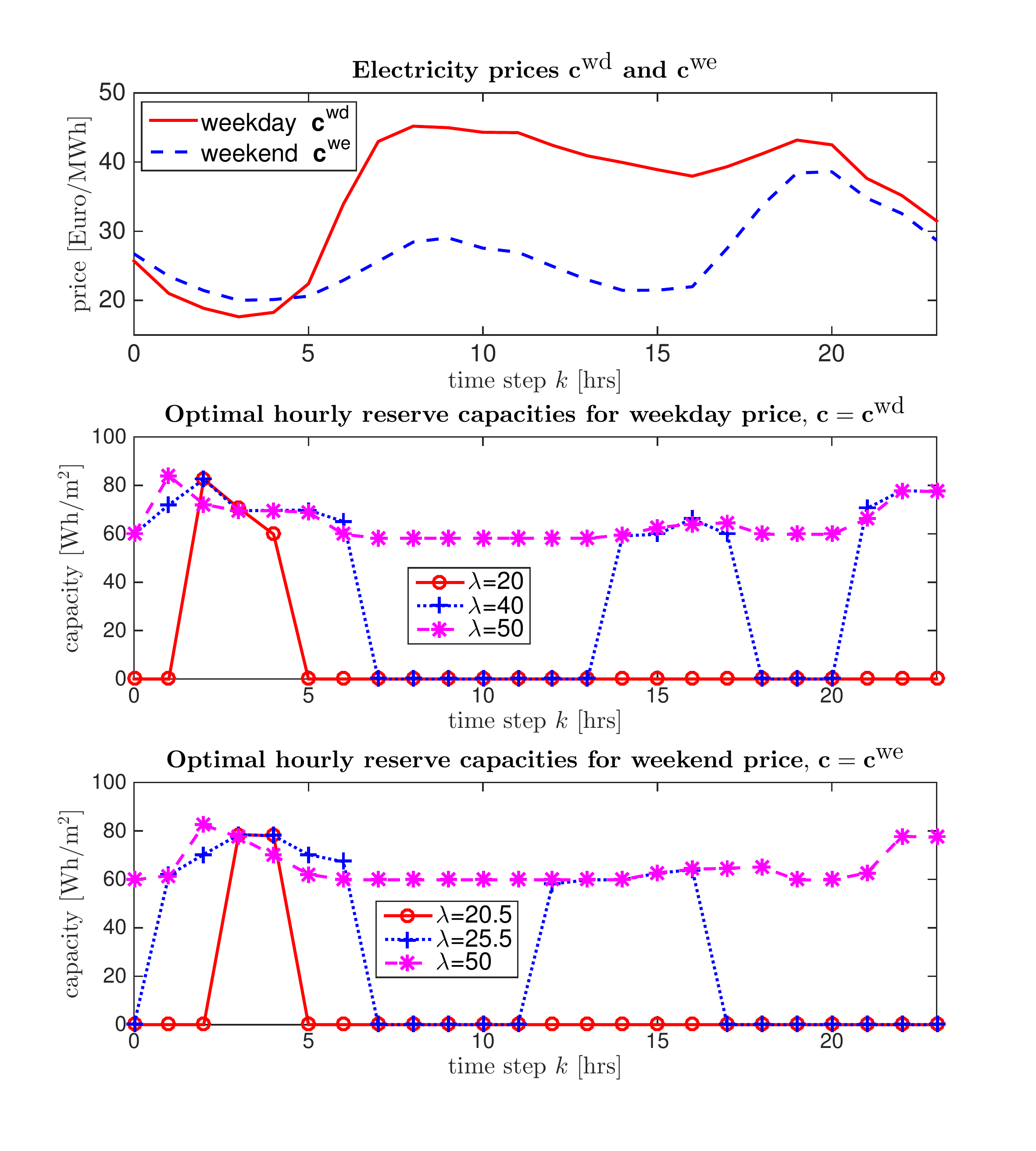}
        \vspace{-0.2cm}
        \caption{\emph{(Top)} Hourly day-ahead spot market prices for weekday (solid red line) and weekend (dashed blue line). The weekday prices are averaged over the workdays 14th--18th\ September 2015, and the weekend prices averaged over 19th--20th September 2015. Data were obtained from \url{www.epexspot.com}. \emph{(Center)} Optimal reserve capacities $Y_k^\star$ for weekday prices and reward factors  $\lambda\in\{20,40,50\}$. \emph{(Bottom)} Optimal reserve capacities $Y_k^\star$ for weekend prices and reward factors  $\lambda\in\{20.5,25.5,50\}$.}
        \label{fig:elecPrices_reservesStages}
\end{figure}

\begin{figure}[t!]
        \centering
        \includegraphics[trim = 10mm 13mm 20mm 11mm, clip, width=0.70\textwidth]{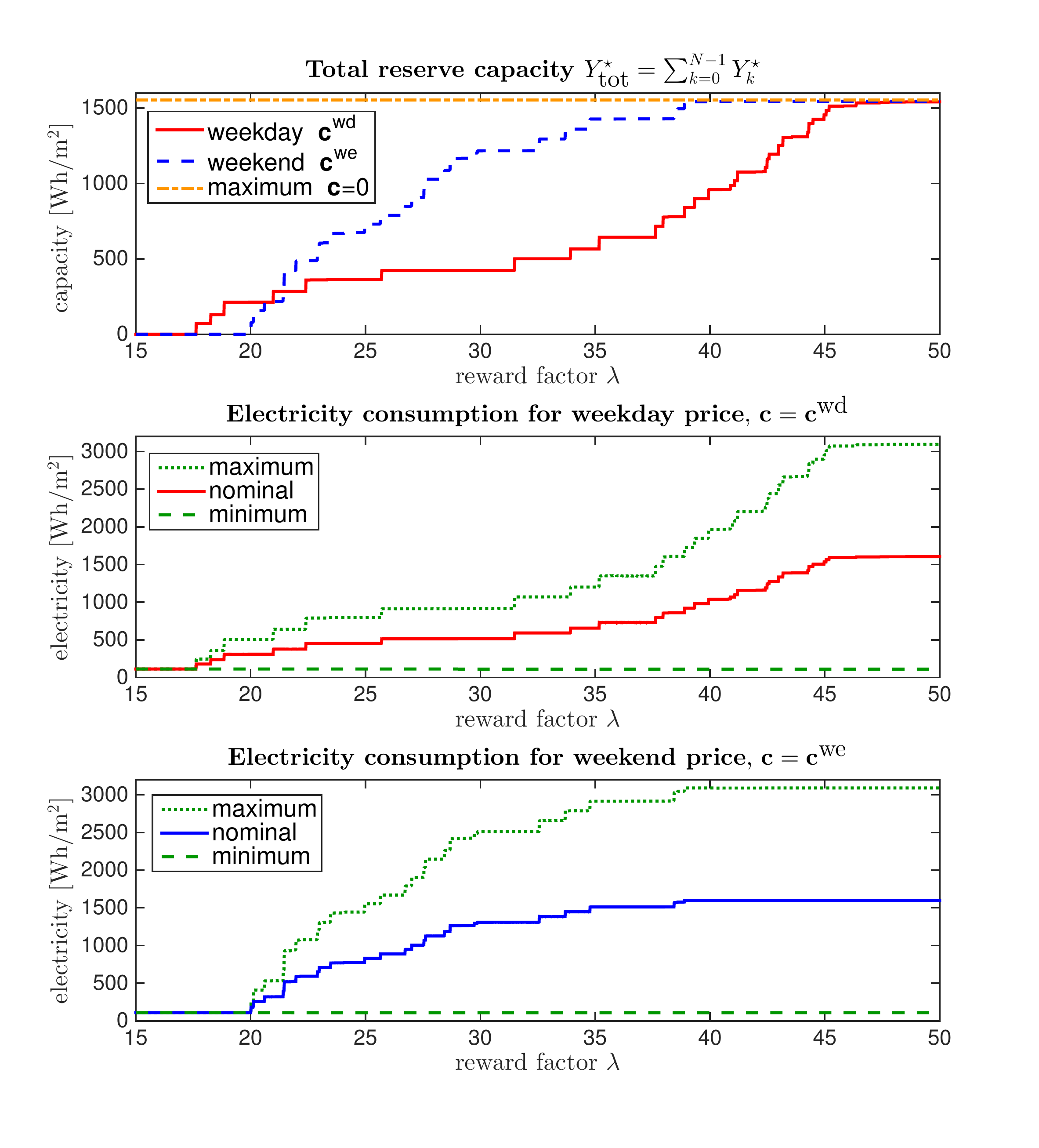}
        \vspace{-0.2cm}
        \caption{\emph{(Top)} Total reserve capacity $Y^\star=\sum_{k=0}^{N-1} Y_k^\star$ for weekday (solid red line) and weekend (dashed blue line). The orange dash-dotted line is the maximum possible offered reserve capacity, obtained by setting $\mathbf{c}=0$. \emph{(Center)} Nominal electricity consumption (solid red line), minimal electricity consumption (dashed green line), and maximum electricity consumption (dotted green line) for weekday price. \emph{(Bottom)} Same as center, but with weekend price.}
        \label{fig:totalReserves_elecConsumption}
\end{figure}

\subsection{Simulation Results}
We consider a one-room building model with one-hour discretization and horizon length $N=24$, based on the reduced building model presented in \cite[Section 4.5]{Oldewurtel_PhD2011}. The states are $x_k = [x_{k,1}, x_{k,2}, x_{k,3}] \in\R^3$, where $x_{k,1}$ is the room temperature, $x_{k,2}$ the inside wall temperature, and $x_{k,3}$ the outside wall temperature. We consider a deterministic sequence for the uncontrolled inputs $v_k=[v_{k,1},v_{k,2}, v_{k,3}]\in\R^{3}$, taken from \cite{zhang:schildbach:sturzenegger:morari:13}. In particular, $v_{k,1}$ represents the outside air temperature, $v_{k,2}$ the solar radiation, and $v_{k,3}$ the number of occupants in the building. We consider four inputs $u_k = [u_{k,1}, u_{k,2}, u_{k,3}, u_{k,4}] \in\R^4$, representing actuators commonly found in Swiss office buildings: a radiator heater, a cooled ceiling, a floor heating system, and mechanical ventilation \cite{Oldewurtel_PhD2011,Oldewurtel:energyBuildings:2012, zhang:schildbach:sturzenegger:morari:13}. Comfort constraints are given as $21^\circ C \leq x_{k,1} \leq 25^\circ C$ for all $k=0,\ldots,N-1$, while the inputs are upper- and lower-bounded\footnote{Often, no temperature constraints are imposed in office buildings during nights. Although not considered here, this feature can easily be incorporated in our problem setup.}. Moreover, we consider two electricity prices $\mathbf{c}^\text{wd}$ and $\mathbf{c}^\text{we}$, where the former exhibits a typical weekday price pattern and the latter a typical weekend price pattern. The prices are depicted in Figure~\ref{fig:elecPrices_reservesStages} (top).

Problem \eqref{eq:example_demandResponse} is an instance of \eqref{eq:OptProblem_detGeneral}, and can be (inner) approximated  based on the ideas discussed in Sections \ref{sec:modUncerSet} and \ref{sec:decisionRuleApprox}. First, following Section \ref{sec:familiesUncertaintySets}, we express $\mathcal{W}=\mathbf{Y}\mathcal{S}$ with $\mathbf{Y}:=\text{diag}(Y_0,\ldots,Y_{N-1})$ and $\mathcal{S}  := \{\mathbf{s}\in\R^N: -\mathbf{1}  \leq \mathbf{s} \leq \mathbf{1}\}$, where $\mathbf{1}\in\R^{N}$ is the all-one vector. Second, we restrict the input policies to admit an affine structure with respect to the primitive set $\mathcal S$ as in \eqref{eq:det_ADF}.

In our first numerical study, we examine the effect of electricity price and reward factor on the reserve capacity provided by the building. In particular, we solve problem~\eqref{eq:example_demandResponse} using the weekday prices with $\lambda\in\{20,40,50\}$, and using the weekend prices with $\lambda\in\{20.5,25.5,50\}$. The results are reported in Figure~\ref{fig:elecPrices_reservesStages} (center and bottom), which show the optimal reserve capacities $[Y_0^\star,\ldots,Y_{N-1}^\star]$ as a function of time. As we can see from the figures,  for a given reward factor $\lambda$, the building  provides reserve capacities only at times during which the reward factor is higher than the electricity price, i.e.\ whenever $\lambda>c_k$. This behavior can be explained as follows: For a building to provide a certain reserve capacity $Y_k$, it must increase its nominal electricity consumption by at least $Y_k$, or else it would not be able to reduce its electricity consumption by $Y_k$. However, since an increase of electricity consumption will incur an additional cost of $c_k Y_k$, the building will only provide reserve capacity of size $Y_k$ if the reward for doing so outweighs the additional electricity cost, i.e.\ whenever $\lambda Y_k > c_k Y_k$. Otherwise, if  $\lambda Y_k < c_k Y_k$, the building will choose not to provide reserves, resulting in the jumps in the reserve capacity profile shown in Figure~\ref{fig:elecPrices_reservesStages} (center and bottom). 
If non-symmetric reserve capacities were allowed, then the building would provide (positive) reserves $\mathbb W_k = [0, \bar Y_k]$ when $\lambda<c_k$. During these times, the building is able to increase its electricity consumption by $w_k=\eta^\top \Delta u_k$, without the need to increase its nominal electricity consumption $e_k=\eta^\top u_k$, since  $w_k\geq0$.
Finally, we point out that, due to the dynamics of the buildings and the comfort constraints, the level of reserves provided can change as $\lambda$ takes on different values.
For example, we see from Figure~\ref{fig:elecPrices_reservesStages} that the reserve profile with $\lambda=50$ is not always higher than that with $\lambda=40$ (center figure) or $\lambda=25.5$ (bottom figure), even if the total reserve capacity turns out to be higher, as we see later on. 

%\red{Although not considered here, our framework also allows for non-symmetric capacities, which can be obtained by  replacing $\mathcal{W}=\mathbf Y\mathcal S$ with $\mathcal{W}=\mathbf Y\mathcal S+\mathbf y$, subject to $-\mathbf Y \leq \text{diag}(\mathbf y) \leq \mathbf Y$, later on.} If non-symmetric reserve capacities were allowed, then the building would provide (positive) reserves $\mathbb W_k = [0, \bar Y_k]$ when $\lambda<c_k$. During these times, the building is able to increase its electricity consumption by $w_k=\eta^\top \Delta u_k$, without the need to increase its nominal electricity consumption $e_k=\eta^\top u_k$, since  $w_k\geq0$. 

In our second numerical study, we study the effect of the reward factor $\lambda$ on the total reserve capacity $Y_\text{tot}^\star := \sum_{k=1}^{N-1}Y_k^\star$ offered by the system. To this end, we solve problem~\eqref{eq:example_demandResponse} with weekday and weekend prices for reward factors $\lambda\in[15,50]$. The results are depicted in Figure~\ref{fig:totalReserves_elecConsumption}~(top), which compares the total reserves for $\mathbf{c}=\mathbf{c}^\text{wd}$ (solid red line), $\mathbf{c}=\mathbf{c}^\text{we}$ (dashed blue line), and the maximum possible capacity obtained when $\mathbf{c}=\mathbf{0}$ (dash-dotted line). Figure~\ref{fig:totalReserves_elecConsumption}~(top) allows us to make the following two observations: First, for most values of $\lambda$, the building provides more reserves during the weekend than during weekdays. This is intuitive because the electricity price during the weekend is typically lower than during weekday, compare with Figure~\ref{fig:elecPrices_reservesStages}~(top). Second, when $\lambda$ is small, more reserves are provided during weekdays, which can be explained by the fact that the minimum electricity price is lower on weekdays than during the weekend, incentivizing the building to provide reserves even for small values of $\lambda$. From a practical point of view, Figure~\ref{fig:totalReserves_elecConsumption}~(top) can be used to determine the optimal level of reserves that should be offered for a given reward $\lambda$. Alternatively, in case of conditional bidding (i.e.\ the amount of reserve bid in the auction is a function of the reward), Figure~\ref{fig:totalReserves_elecConsumption}~(top) can be used to determine the optimal tradeoff between the amount of reserves and the reward factor.

In our final numerical study, we investigate the effect of the reward factor $\lambda$ on the building's electricity consumption. To this end, we define the nominal, maximal and minimal electricity consumption as $\bm{\eta}^\top ( \bm\pi^\star(\mathbf{0})+\bm\kappa^\star(\mathbf 0))$, $\max_{\mathbf{w}\in\mathcal{W}^\star}\{\bm\eta^\top ( \bm\pi^\star(\mathbf w)+\bm\kappa^\star(\mathbf w))\}$, and $\min_{\mathbf{w}\in\mathcal{W}^\star}\{\bm\eta^\top ( \bm\pi^\star(\mathbf w)+\bm\kappa^\star(\mathbf w))\}$,  respectively, where $\bm\pi^\star(\cdot)$ and $\bm\kappa^\star(\cdot)$ are the optimal affine policies computed from problem~\eqref{eq:example_demandResponse}, and $\mathcal{W}^\star$ is the associated optimal reserve capacity. The results are depicted in Figure~\ref{fig:totalReserves_elecConsumption}~(center and bottom). It can be seen that as the building provides more reserves for increasing $\lambda$, it must also increase its nominal electricity consumption. Indeed, as can be verified, the nominal electricity consumption increases in the exact same fashion as the total reserve capacity depicted in Figure~\ref{fig:totalReserves_elecConsumption}~(top). Similar to what we have seen before, such behavior is to be expected since to provide a certain level of reserves, the building must increase its nominal electricity consumption by the same amount. Moreover, the interval between the maximal and minimal electricity consumption can be interpreted as the ``slack" of the building, within which the building can adjust its electricity consumption without violating its constraints.

%\blue{Recall that affine policies are known to be suboptimal in general. Nevertheless, by comparing our results with the optimal solution obtained from the scenario approach \cite{Scokaert1998}, it turns out they are equivalent if $N\leq8$, for all $\lambda\in[15,50]$. This gives hope that affine policies are optimal for our reserve provision problem, although we cannot verify it numerically for $N>8$ due to the exponential growth of the problem size of the scenario approach.} 

We finish this section with three concluding remarks: $(i)$ Our framework also allows for non-symmetric reserve capacities. For example, non-symmetric reserve capacities of the form $[-\underline{Y}_k,\overline{Y}_k]$ can be obtained by  replacing $\mathcal{W}=\mathbf Y\mathcal S$ in \eqref{eq:example_demandResponse} with $\mathcal{W}=\mathbf Y\mathcal S+\mathbf y$ and $-\mathbf Y \leq \text{diag}(\mathbf y) \leq \mathbf Y$. Similarly, positive reserve capacities are obtained by choosing the primitive set $\mathcal{S}$ in \eqref{eq:example_demandResponse} as $\mathcal{S}=[0,1]^N$, or as $\mathcal{S}=[-1,0]^N$ if negative reserves are desired.
$(ii)$~It is well-known that restricting the input policies to be affine can be conservative, and  that measuring this loss of optimality exactly is as challenging as determining the optimal policy. Nevertheless, for small horizons, the solution to problem \eqref{eq:example_demandResponse} can be computed exactly by enumerating all corner points of $\mathcal{W}$, and then solving the corresponding optimization problem \cite{Scokaert1998}. Unfortunately, this approach is only applicable if $N$ is small, since the problem size grows exponentially in $N$ \cite{Scokaert1998}. Nevertheless, if $N\leq8$ in \eqref{eq:example_demandResponse}, by using this enumeration approach, we verified that affine policies are indeed optimal for all $\lambda\in[15,50]$. %Although we cannot make a definitive statement for larger $N$, this result gives hope that affine policies are indeed optimal in this case.
\blue{For longer horizons $9\leq N \leq 24$, a bound on the loss of optimality can be obtained based on the method introduced in \cite{BertsimasGeorghiouOR2015}. Using this approach, we have verified that the relative gap between the exact solution and the affine policies is consistently below 0.71\%, for all $\lambda\in[15,50]$. This demonstrates that the affine policy approximation achieves a high degree
of precision for this example, giving us hope that affine policies are indeed optimal in this case.} 
$(iii)$~We highlight that under the affine policy restriction, problem~\eqref{eq:example_demandResponse} with $N=24$ can be written as a linear optimization problem with roughly 13\,000 decision variables and 18\,000 constraints. For fixed $\lambda$, each problem was solved within 0.3 seconds using the optimization solver Mosek on a standard laptop equipped with 16 GB RAM and a 2.6 GHz quad-core Intel i7 processor, demonstrating that our proposed solution method is indeed computationally tractable.

\section{\blue{Example 2: Robustness Analysis}}\label{sec:example2}
We next apply our framework to a robustness analysis problem, where the goal is to determine the largest uncertainty set a system can tolerate without violating its constraints. We consider a system of the form \eqref{eq:systemDynamics_generic} with parameters $n_x=2$, $n_u=1$, $n_w=2$, and system matrices $A=I$, $B = [1, 0.7]$ and $E=-I$, where $I\in\R^{2\times2}$ is the identity matrix. To simplify the subsequent analysis, we set $N=1$, $x_0 =\bar x=0$, and $u_0=\pi_0(w_0)\in\mathcal C$. Hence, the robustness analysis problem can be written as a robust optimal control problem with adjustable uncertainty set:

\begin{equation}\label{eq:example_maxUncSet}
\begin{array}{ll}
\hspace{-0.1cm} 	\text{max}  & \text{vol}(\mathbb {W}) \\[0.3ex]
\hspace{-0.05cm} 	\text{s.t.}     	&  \pi(\cdot)\in\mathcal{C}, \ \mathbb W\in\mathcal P(\R^2),\ \ \bar x=0, \\[0.3ex]
                        	& \!\!\!\left. 
                                \begin{array}{l}        x(w) = A \bar x + B\pi(w) - w,\\[0.3ex]
                                                                x(w) \in \mathcal X,\ \pi(w) \in \mathcal U, 
                                        \end{array}   \hspace{-0.0cm} \right\} \forall w\in\mathbb{W},
\end{array}
\end{equation}
with optimization variables $(\mathbb W,\pi(\cdot))$, where, for simplicity, we have omitted the time indices in our one-step problem. The state and input constraints are defined as $\mathcal{X} := \{[x_1,\, x_2]\in\R^2: [-10,\, -10] \leq [x_1,\, x_2] \leq [10,\, 10] ,\ -15\leq -x_1+x_2\leq 15,\ -15 \leq x_1+x_2 \leq 15   \}$ and $\mathcal{U}:=\{u\in\R:-5 \leq u\leq 5\}$. Problem \eqref{eq:example_maxUncSet} is an instance of \eqref{eq:OptProblem_detGeneral}, and can thus be dealt with by first expressing the uncertainty set as $\mathbb W = Y\mathbb S + y$, where $\mathbb S$ is a primitive set, and then restricting the input policy as $\tilde\pi(\cdot)\in\tilde{\mathcal C}_\text{aff}$. Finally, a suitable concave objective function $\varrho(\cdot)$ is chosen that approximates the objective function $\text{vol}(\cdot)$.

%\begin{table*}[t]
%\centering
%\begin{tabular}{c c|c|c}
%\toprule
%\emph{Examples} & constraint function $g(\cdot,\cdot)$ & (relevant) problem data & bound on $\zeta$ \\ 
%\midrule
%example 1:  & $A G(x)q(\delta)+Bh(x)+s(\delta)$ & $q(\delta)\in\R^p$ & $\rank(A) p + \rank(B)$ \\    
%
%example 2: & $q(\delta)^\top A h(x) $ &  -- & $\rank(A)$ \\    
%
%example 3: & $Bh(x)-s(\delta) $ &  -- & $\rank(B)$ \\    
%
%example 4: & $Bh(x)-\mathbf{1}\delta$ &  $\delta\in\Delta\subset\R$ & 1 \\    
%
%example 5: & $(Ax-s(\delta))^\top Q (Ax-s(\delta))$ &  $L^\top L = Q\succeq0$ & $\rank(LA)$ \\    
%
%example 6: & $\exp\{\delta \|Ax_1-b\|\} - x_2$ &  $\delta\in\R_+$ & 2 \\    
%
%\bottomrule
%\end{tabular}
%\end{table*}

\begin{table}[b]
\tabcolsep = 1.5mm
\caption{Optimal value of different approximation schemes.}
\vspace{-0.5cm}
\label{tab:objectiveValue}
\begin{center}
\begin{tabular}{l|c|c }
\toprule
 Set Family & $\text{vol}(\mathbb W)$ & Approx.\ Quality\\
 \midrule
$(i)$\ \ \ $\mathbb W_\text{rect}^\star \in \Omega_\text{rect}$ &  260.4 & 42.0\% \\
$(ii)$\ \ $\mathbb W_\text{ell}^\star \in \Omega_\text{ell}$ &  514.4 & 83.0\% \\
$(iii)$ $\mathbb W_\text{poly}^\star \in \Omega_\text{poly}$ &  620.2 & 100\%  \\ 
%\midrule
%Optimal Solution &  620.2 & 100\% \\
\bottomrule
\end{tabular}
\end{center}
\end{table}

\begin{figure*}[t!]
        \centering
        \includegraphics[trim = 60mm 0mm 40mm 5mm, clip, width=1\textwidth]{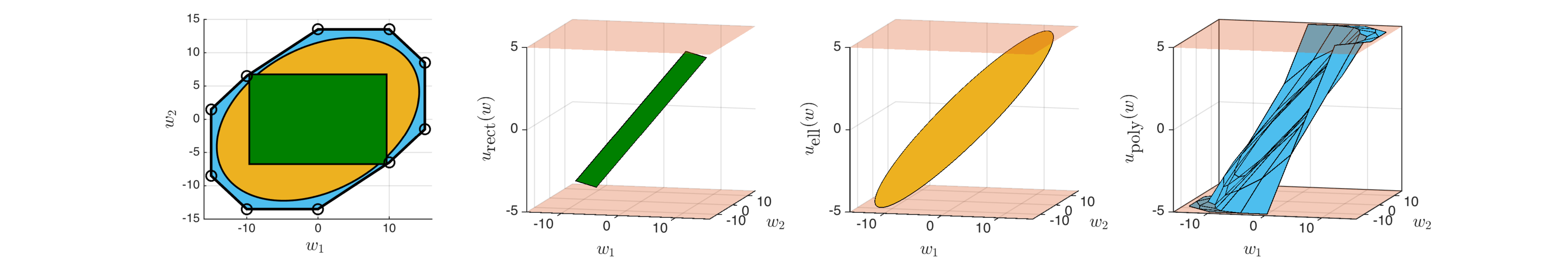}
        \vspace{-0.2cm}
        \caption{\emph{(Left)} The optimal uncertainty sets $\mathbb W_\rect^\star$ (green rectangle), $\mathbb W_\ellipse^\star$ (yellow ellipse) and $\mathbb W_\poly^\star$ (blue polytope). \emph{(Center left)} Affine policy $\pi_\rect(\cdot)$ for rectangular uncertainty set recovered from $\tilde\pi^\star(\cdot)\in\tilde{\mathcal C}_\text{aff}$. \emph{(Center right)} Affine policy $\pi_\ellipse(\cdot)$ for ellipsoidal uncertainty set recovered from $\tilde\pi^\star(\cdot)\in\tilde{\mathcal C}_\text{aff}$. \emph{(Right)} Piece-wise affine policy $\pi_\poly(\cdot)$ for polytopic uncertainty set recovered from $\tilde\pi^\star(\cdot)\in\tilde{\mathcal C}_\text{aff}$.}
        \label{fig:maxUncSet_Policy}
\end{figure*}

In our first study, we examine how the choice of the uncertainty set family affects the solution quality. The following three families of uncertainty sets, all discussed in Section~\ref{sec:familiesUncertaintySets}, are considered: $(i)$ rectangular uncertainty sets $\mathbb W_\text{rect}\in\Omega_\text{rect}$; $(ii)$ ellipsoidal uncertainty sets $\mathbb W_\text{ell}\in\Omega_\text{ell}$; and $(iii)$ polytopic uncertainty sets $\mathbb W_\text{poly}\in\Omega_\text{poly}$ with 30 vertices.  Following Section~\ref{sec:volMaximization}, we use $\varrho_\text{rect}(Y\mathbb S+y) = \varrho_\text{ell}(Y\mathbb S+y)=\log\det(Y)$ to maximize $\text{vol}(\mathbb W_\text{rect})$ and $\text{vol}(\mathbb W_\text{ell})$, respectively. Since we cannot directly maximize the volume of the polytope $\mathbb W_\text{poly}$, we use the quadratic cost function $\varrho_\text{poly}(Y\mathbb S)=-\sum_{j=1}^{n_s} \| d_j^\top - Y_{\cdot j}\|_2^2$ described in Section~\ref{sec:familiesUncertaintySets} to approximate $\text{vol}(\mathbb W_\text{poly})$, where $n_s=30$, and $d_j$ are uniformly distributed on a centered circle of radius forty, i.e.\ $d_j := [40\cos(2\pi(j-1)/n_s),\ 40\sin(2\pi(j-1)/n_s)  ]$.  Figure~\ref{fig:maxUncSet_Policy} (left) depicts the optimal uncertainty sets $\mathbb W_\text{rect}^\star$, $\mathbb W_\text{ell}^\star$, and $\mathbb W_\text{poly}^\star$, while Table~\ref{tab:objectiveValue} lists their volumes.  For this example, we see that $\vol(\mathbb W_\rect^\star)$ is much smaller than $\vol(\mathbb W_\ellipse^\star)$ and $\vol(\mathbb W_\poly^\star)$, because $\Omega_\text{rect}$ restricts the sets to be axis-aligned, whereas both $ \Omega_\text{ell}$ and $ \Omega_\text{poly}$ allow for rotated and skewed sets. Moreover, it is interesting to observe that even though $\varrho_\poly(\cdot)$ does not maximize the volume of $\mathbb W_\poly$ explicitely, it yet yields the largest uncertainty. For this example, it can be verified by taking all combinations of the vertices of the sets $\mathcal X$ and $\mathcal U$ %exhaustive enumeration of the vertices of the set $\mathcal X\times\mathcal U$ 
that the polytopic approximation method returns the optimal solution to problem~\eqref{eq:example_maxUncSet}, i.e.\ $\mathbb W^\star = \mathbb W_\poly^\star$. Observe now from Figure~\ref{fig:maxUncSet_Policy} (left) that while $\mathbb W_\text{rect}^\star$ corresponds to the largest volume rectangle contained in $\mathbb W^\star$, $\mathbb W_\text{ell}^\star$ is \emph{not} the maximum volume ellipse contained in $\mathbb W_\text{poly}$. As we will see in the following paragraph, this is due to the requirement that $\tilde\pi_\ellipse(\cdot)$ is affine. 
%\begin{figure}[t!]
%        \centering
%        \includegraphics[trim = 0mm 0mm 0mm 0mm, clip, width=0.45\textwidth]{Figures/maxUncSet1}
%        \vspace{-0.2cm}
%        \caption{Optimal solutions for the three uncertainty set families $(i):\mathbb W\in\Omega_\text{rect}$, $(ii):\mathbb W\in\Omega_\text{ell}$, and $(iii):\mathbb W\in\Omega_\text{poly}$. The polytope approximation yields the optimal solution for problem~\eqref{eq:example_maxUncSet}.}
%        \label{fig:maxUncSet}
%\end{figure}

In the second part of this study, we examine the structure of the policy $\pi(\cdot)$ for the different uncertainty sets. Figure~\ref{fig:maxUncSet_Policy} (center left -- right) depicts the policies $\pi(\cdot)$ for all three cases, where we recovered $\pi_\rect(\cdot)$ and $\pi_\ellipse(\cdot)$ using \eqref{eq:recoveringW}, and $\pi_\poly(\cdot)$  using \eqref{eq:recoveringW_pwa} with $L(\cdot)$ in \eqref{L}. We see from Figure~\ref{fig:maxUncSet_Policy} (center right) that the reason $\mathbb W_\ellipse^\star$ is not the maximum volume ellipse is due to the restriction $\pi_\ellipse(\cdot)\in\mathcal C_\text{aff}$. Indeed, we  observe that even for $\mathbb W_\ellipse=\mathbb W_\ellipse^\star$, the associated policy $\pi_\ellipse(\cdot)$ hits the input constraints, and that if $\mathbb W_\ellipse$ were expanded in the $45^\circ$-direction, then the input constraints would be violated. This drawback, in principle, can be alleviated by allowing $\pi_\ellipse(\cdot)$ to be non-linear, using e.g.\ the techniques in \cite{georghiouGeneralized2015, GeorghiouBinaryMP2015}. Note that this issue is not present in the polytopic approximation, since the policy $\pi_\poly(\cdot)$ is piece-wise affine by construction, c.f.\ Proposition~\ref{prop} $(iii)$. Indeed, a careful inspection of $\pi_\poly(\cdot)$ reveals that it is the piece-wise affine nature that allows $\mathbb W_\poly^\star$ to obtain its complex shape with vertices that are not captured by either $\mathbb W_\rect^\star$ or $\mathbb W_\poly^\star$, and which gives rise to the uncertainty set with biggest volume.

%\begin{figure}[t!]
%        \centering
%        \includegraphics[trim = 0mm 0mm 0mm 0mm, clip, width=0.45\textwidth]{Figures/PWAPolicy}
%        \vspace{-0.2cm}
%        \caption{The optimal policy $u_0^\star(w_0)$ associated with approximation $(iii)$. On each ``piece", the policy is affine.}
%        \label{fig:PWAPolicy}
%\end{figure}

%%%%%%%%%%%%%%%%%%%%%%%%%%%%%%%%%%%%%%%%

\section{Conclusion}\label{sec:conclusion}
This paper presents a unified framework for studying robust optimal control problems with adjustable uncertainty sets. We have shown that these problems can be approximated as a convex optimization problem if $(i)$ the uncertainty sets are representable as affine transformations of primitive sets, and $(ii)$ the control policies are restricted to be affine with respect to the primitive sets. Under these restrictions, the resulting convex optimization problem grows polynomially in the size of the problem data and the description of the primitive set, resulting in a tractable optimization problem. The applicability and importance of such problems was demonstrated through a reserve capacity problem arising in power systems. A numerical example confirms that these problems can be solved efficiently, and that optimal bidding and reserve strategies can be derived.
%managerial insights can be drawn for the optimal bidding strategies of the buildings.

Future work will focus on reserve provision problems, addressing stability and recursive feasibility when the reserve capacities are bid in a receding horizon fashion. Moreover, we plan on investigating the optimality of affine policies in the reserve provision context, as well as the impact external uncertainties have on the provided reserves. Finally, the presented framework could be extended to cases when only partial state observations are available.

\balance 
\section*{Acknowledgments}
%The authors thank Paul Goulart for inspiring discussions. 
This research was partially supported by Nano-Tera.ch under the project HeatReserves and by the European Commission under
the project Local4Global.

% ============================================================================
\begin{appendix}
%\section{Appendix}
%\subsection{Definition of Matrices \red{[To be fixed!]}}
%\begin{equation*}
%        \mathbf{H} := \begin{pmatrix}   H_{0,0} & 0 & \cdots & 0 \\
%                                                        H_{1,0} & H_{1,1} & \cdots & 0 \\
%                                                        \vdots & \vdots & \ddots & \vdots \\
%                                                        H_{N-1,0} & H_{N-1,1} & \cdots & H_{N-1,N-1}
%                                                          \end{pmatrix}, \quad
%        \mathbf{E} := \begin{pmatrix}   E & 0 & \cdots & 0 \\
%                                                        AE & E & \cdots & 0 \\
%                                                        \vdots & \vdots & \ddots & \vdots \\
%                                                        A^{N-1}E & A^{N-2}E & \cdots & E
%                                                          \end{pmatrix},\quad
%        \mathbf{A} :=   \begin{pmatrix}  A \\ A^2 \\ \vdots \\ A^{N-1}   \end{pmatrix}
%\end{equation*}
%\begin{align*}
%        & \bm\Sigma := \text{blkdiag}(\Sigma_0,\ldots,\Sigma_{N-1}), \mathbf{D} := \text{blkdiag}(D_0,\ldots,D_{N-1}), \mathbf{d} = [d_0^\top, \ldots, d_{N-1}^\top]^\top \\
%        & \mathbf{G} := \mathbb{I}_{N\times N}\otimes G,  \mathbf{g} := \mathbb{I}_{N\times 1}\otimes g, \mathbf{F} := \mathbb{I}_{N\times N}\otimes G,  \mathbf{f} := \mathbb{I}_{N\times N}\otimes f
%\end{align*}

%%%%%%%%%%%%%%%%%%%%%%%%%%%%%%%%%%%%%%%%%%%%%%%%%%%%
%%%%%%%%%%%%%%%%%%%%%%%%%%%%%%%%%%%%%%%%%%%%%%%%%%%%
\section{Definition, Lemma and Proofs}
In this section, we provide the proofs for Propositions~\ref{prop:equiv_reparaPolicies}  and \ref{prop:relationshipPolicyW_and_PolicyS}. To do this, we first need to introduce the \emph{lifting operator}, and an auxiliary result given in Lemma~\ref{lem:liftingOperators}. We will use similar theoretical tools as those developed in \cite{georghiouGeneralized2015, GeorghiouBinaryMP2015}.

%To prove Propositions \ref{prop:equivalence_simpleProblem}--\ref{prop:relationshipPolicyW_and_PolicyS}, we need the following axiomatic definition of %a \emph{lifting operator}\footnote{We use here the word ``lifting" as we typically think of $\mathcal{S}$ to live in a higher dimension than $\mathcal{W}$ %with $n_s\geq n_w$.}.
\begin{definition}[Lifting operator]\label{def:Lifting}
Given $(\mathbf{Y},\mathbf{y})\in\mathcal{Y}$ and $\mathcal{S} = \mathbb{S}_0 \times \ldots \times \mathbb{S}_{N-1}$, where each $\mathbb S_k$ is a primitive set, let $\mathcal{W}:=\mathbf{Y}\mathcal{S}+\mathbf{y}$. We call a mapping $L: \mathcal{W}\to\mathcal{S}$ a \emph{lifting operator} if it satisfies the following  properties: 
 \begin{itemize}
        \item[(P1)] $L(\cdot)= [L_0(\cdot),\ldots,L_{N-1}(\cdot)]$, where, for all $k=0,\ldots,N-1$, $L_k:\mathbb{W}_k \to \mathbb{S}_k$;
        \item[(P2)] $L_k(w) \in \mathbb{S}_k,$ for all $w\in\mathbb{W}_k$;
        \item[(P3)] $\mathbf{Y} L(\mathbf w) + \mathbf y = \mathbf w,$ for all $\mathbf w\in\mathcal{W}$. 
        \end{itemize}
\end{definition}
We refer to $L(\cdot)$ as the  lifting operator  since typically $\text{dim}(\mathcal S)\geq \text{dim}(\mathcal W)$, i.e.\ $n_s\geq n_w$.  Properties (P1) and (P2) ensure that $L(\cdot)$ preserves causality since at time $k$, since $[L_0(\cdot),\ldots,L_{k}(\cdot)]$ contains information up to and including stage $k$ only. 
Notice that, although $L(\cdot)$ is defined through $(\mathbf{Y,y})$ and $\mathcal{S}$, there can be multiple choices of $L(\cdot)$ that satisfy \textnormal{(P1)}--\textnormal{(P3)}. The following example demonstrates the existence of such a lifting for given $(\mathbf{Y,y})\in\mathcal{Y}$, $\mathcal{S}$, and $\mathcal{W}:=\mathbf Y \mathcal S + \mathbf y$, while highlighting some of its properties.
%In the  following, we give an example of $L(\cdot)$ and highlight some of its properties.
\begin{example}\label{example:Lifting}
Given $(\mathbf{Y},\mathbf{y})\in\mathcal{Y}$, $\mathcal{S} = \mathbb{S}_0 \times \ldots \times \mathbb{S}_{N-1}$ and
$\mathcal{W}:=\mathbf{Y}\mathcal{S}+\mathbf{y}$, let us define $L(\cdot):= [L_0(\cdot),\ldots,L_{N-1}(\cdot)]$ through its components
\begin{equation}\label{L}
L_k(w):=\arg\min_{z\in\mathbb{S}_k}\{\|z\|_2^2: Y_k z+y_k = w\},
\end{equation}
for all $w\in\mathbb W_k$. It is easy to verify that  this lifting operator satisfies conditions \textnormal{(P1)}--\textnormal{(P3)}.
Moreover, if $n_s= n_w$ and $\mathbf{Y}$ is a full rank matrix, then $L(\cdot)$ is the bijective operator $L(\mathbf{w}) = \mathbf{Y}^{-1}(\mathbf{w}-\mathbf{y})$. However, if $\mathbf{Y}$ is rank deficient, then  the mapping from $\mathcal{W}$ to $\mathcal{S}$ is not unique, and therefore, operator \eqref{L} chooses the mapping that minimizes the $2$-norm in the $\mathcal{S}$ space. Finally, note that the above optimization problem defining $L(\cdot)$ is always feasible by construction, and that alternative liftings can be defined using different (also non-linear) objective functions.
\end{example}

The following lemma establishes the relation between  constraint functions in the $\mathcal{W}$-space and  in the $\mathcal{S}$-space.
\begin{lemma}\label{lem:liftingOperators}
Given $(\mathbf{Y},\mathbf{y})\in\mathcal{Y}$, $\mathcal{S} = \mathbb{S}_0 \times \ldots \times \mathbb{S}_{N-1}$, where each $\mathbb S_k$ is a primitive set, let 
        $\mathcal{W}:=\mathbf{Y}\mathcal{S}+\mathbf{y}$ and $L(\cdot)$ be a lifting operator. For any two functions $g:\R^{Nn_w}\to\R$ and $f:\R^{Nn_s}\to\R$, it holds:
%\begin{itemize}
%\item[(i)] \qquad $g(\mathbf{w})\leq0, \ \forall\mathbf{w}\in\mathcal{W}(\mathbf{Y,y})\quad \\ \Longleftrightarrow \quad  g(\mathbf{Y}\mathbf{s}+\mathbf{y})\leq0 ,\ \forall\mathbf{s}\in\mathcal{S}$;
%\item[(ii)] $f(\mathbf{s}) \leq 0,\;\;\forall\mathbf{s}\in\mathcal{S}\ \Longrightarrow\ f(L(\mathbf{w})) \leq 0,\ \forall\mathbf{w}\in\mathcal{W}(\mathbf{Y,y})$.
%\end{itemize}
\begin{itemize}

\item[(i)]  $g(\mathbf{w})\leq0,\  \forall\mathbf{w}\in\mathcal{W}\quad\Longleftrightarrow \quad  g(\mathbf{Y}\mathbf{s}+\mathbf{y})\leq0 ,\ \forall\mathbf{s}\in\mathcal{S}$;
%                 \begin{equation*}
%                         \vspace{-0.2cm}
%                         \begin{array}{lllll}
%                                 & g(\mathbf{w})\leq0, & \forall\mathbf{w}\in\mathcal{W}(\mathbf{Y,y})\\
%                                 \Longleftrightarrow & g(\mathbf{Y}\mathbf{s}+\mathbf{y})\leq0 ,\ & \forall\mathbf{s}\in\mathcal{S}.
%                         \end{array}
%                 \end{equation*}
                
%\qquad $g(\mathbf{w})\leq0, \ \forall\mathbf{w}\in\mathcal{W}(\mathbf{Y,y})\quad \\ \Longleftrightarrow \quad  g(\mathbf{Y}\mathbf{s}+\mathbf{y})\leq0 ,\ \forall\mathbf{s}\in\mathcal{S}$;

\item[(ii)] $f(\mathbf{s}) \leq 0,\ \forall\mathbf{s}\in\mathcal{S} \quad \Longrightarrow\quad f(L(\mathbf{w})) \leq 0,\ \forall\mathbf{w}\in\mathcal{W}$.    
% \vspace{-0.2cm}
%                 \begin{equation*}
%                         \vspace{-0.2cm}
%                         \begin{array}{lllll}
%                                 & f(\mathbf{s}) \leq 0, &\forall\mathbf{s}\in\mathcal{S} \\
%                                 \Longrightarrow & f(L(\mathbf{w})) \leq 0,\ &\forall\mathbf{w}\in\mathcal{W}(\mathbf{Y,y}).
%                         \end{array}
%                 \end{equation*}
%$f(\mathbf{s}) \leq 0,\;\;\forall\mathbf{s}\in\mathcal{S}\ \Longrightarrow\ f(L(\mathbf{w})) \leq 0,\ \forall\mathbf{w}\in\mathcal{W}(\mathbf{Y,y})$.
\end{itemize}
\end{lemma}
\begin{proof}
We prove $``\Longrightarrow"$ of  $(i)$ by contradiction. Assume that  $g(\mathbf{w})\leq0$ for all $\mathbf{w}\in\mathcal{W}$, but there exists $\mathbf{s}\in\mathcal{S}$ such that $g(\mathbf{Y\mathbf{s}+y})>0$. Since $\mathbf{Y\mathbf{s}+y}\in\mathcal{W}$ for all $\mathbf{s}\in\mathcal{S}$ by definition of $\mathcal W$, this leads to a contradiction. The \mbox{$``\Longleftarrow"$} part of $(i)$ is clear, since $\mathcal{W}= \mathbf{Y}\mathcal{S}+\mathbf{y}$ by definition.  Statement $(ii)$ follows immediately from property (P2) in the definition of the lifting operator.
\end{proof}

%\section{Proofs}
We now have the necessary tools to prove Propositions~\ref{prop:equiv_reparaPolicies}  and \ref{prop:relationshipPolicyW_and_PolicyS}.
\begin{proof}[Proof of Proposition \ref{prop:equiv_reparaPolicies}]
To simplify notation, we first  rewrite problem \eqref{eq:OptProblem_detGeneral_1} using the  epigraph formulation 
\begin{equation}\label{eq:compact_w}
\hspace{-0.0cm}
\begin{array}{ll}
\text{min} &  \tau - \lambda\bm{\varrho}(\mathcal{W}) \\[0.3ex]
\text{s.t.} &  \tau\in\R,\,\bm{\pi}(\cdot)\in\mathcal{C},\, \mathcal{W}=\mathbf{Y}\mathcal{S}+\mathbf{y},\, (\mathbf{Y},\mathbf{y})\in\mathcal{Y}, \\[0.3ex]
& \!\!\!\left. 
\begin{array}{l}
\mathbf{c}^\top \bm\pi(\mathbf{w}) \leq \tau, \\ [0.3ex]
\mathbf{C} \bm\pi(\mathbf{w}) + \mathbf{Dw}\leq \mathbf{d},                                
\end{array} \right\} \forall\mathbf{w}\in\mathcal{W},
\hspace{-0.0cm}
\end{array}
\hspace{-0.0cm}
\end{equation}
with decision variables $(\tau,\mathcal{W},\bm\pi(\cdot),\mathbf{Y,y})$, and problem \eqref{eq:OptProblem_detGeneral_2} as
\begin{equation}\label{eq:compact_s}
\begin{array}{ll}
\textnormal{min}  &   \tau - \lambda\bm\varrho(\mathbf{Y}\mathcal{S}+\mathbf{y}) \\[0.3ex]
\textnormal{s.t.}         & \tau\in\R,\,\bm{\widetilde\pi}(\cdot)\in\widetilde{\mathcal{C}},\, (\mathbf{Y},\mathbf{y})\in\mathcal{Y},\\[0.3ex]
& \!\!\!\left. 
\begin{array}{l}
\mathbf{c}^\top \bm{\widetilde\pi}(\mathbf{s}) \leq \tau, \\ [0.3ex]
\mathbf{C} \bm{\widetilde\pi}(\mathbf{s}) +\mathbf{D}(\mathbf{Ys+y}) \leq \mathbf{d},                
\end{array} \right\}\;\;\forall\mathbf{s}\in\mathcal{S},
\end{array}
\end{equation}
with decision variables $(\tau,\bm{\widetilde\pi}(\cdot),\mathbf{Y,y})$.

We now show that any feasible solution in \eqref{eq:compact_w} corresponds to a feasible solution in \eqref{eq:compact_s} with the same objective value, and vice versa. Suppose that $(\tau,\mathcal{W},\bm\pi(\cdot),\mathbf{Y},\mathbf{y})$ is a feasible solution to \eqref{eq:compact_w}. We define  $\bm{\widetilde\pi}(\cdot)$ such that $\bm{\widetilde \pi}(\mathbf{s}) :=\bm\pi(\mathbf{Y}\mathbf{s}+\mathbf{y})$, for all $\mathbf{s}\in\mathcal{S}$. Note that due to the block-diagonal form of $\mathbf{Y}$, the constructed policy $\bm{\widetilde\pi}(\cdot)$ is causal,  satisfying the constraint $\bm{\widetilde\pi}(\cdot)\in\widetilde{\mathcal{C}}$. Since $(\tau,\mathcal W,\bm\pi(\cdot),\mathbf{Y},\mathbf{y})$ is feasible in \eqref{eq:compact_w}, it implies that
\begin{equation*}
\begin{array}{rllllll}

        \vspace{0.1cm}
        & \left. \begin{array}{l}
        \mathbf{C}\bm\pi(\mathbf{w}) + \mathbf{Dw} \leq \mathbf{d}, \\[0.3ex]
        \mathbf{c}^\top\bm{\pi}(\mathbf{w})\leq \tau, 
        \end{array} \right\} \forall\mathbf{w}\in\mathcal{W},\\[0.3ex]

        \vspace{0.1cm}
        \Longleftrightarrow & \left. \begin{array}{l}   
        \mathbf{C}\bm\pi(\mathbf{Y}\mathbf{s}+\mathbf{y}) + \mathbf{D}(\mathbf{Ys+y})\leq \mathbf{d}, \\[0.3ex]
        \mathbf{c}^\top \bm\pi(\mathbf{Ys+y})\leq \tau, 
        \end{array} \right\}  \forall\mathbf{s}\in\mathcal{S}, \\[0.3ex]

%       \vspace{0.1cm}
        \Longleftrightarrow & \left. \begin{array}{l}
        \mathbf{C}\bm{\widetilde \pi}(\mathbf{s}) + \mathbf{D(Ys+y)} \leq \mathbf{d}, \\[0.3ex]
         \mathbf{c}^\top\bm{\widetilde\pi}(\mathbf{s})\leq \tau,  
         \end{array} \right\} \forall\mathbf{s}\in\mathcal{S},
\end{array}
\vspace{-0.0cm}
\end{equation*}
%\begin{equation}
%\begin{array}{rllllll}
%&\mathbf{C}\bm\pi(\mathbf{w}) + \mathbf{Dw} \leq \mathbf{d},\,\mathbf{c}^\top\bm{\pi}(\mathbf{w})\leq \tau, & \forall\mathbf{w}\in\mathcal{W}(\mathbf{Y,y}),\\
%\Longleftrightarrow & \mathbf{C}\bm\pi(\mathbf{Y}\mathbf{s}+\mathbf{y}) + \mathbf{D}(\mathbf{Ys+y})\leq \mathbf{d},\,  \mathbf{c}^\top \bm\pi(\mathbf{Ys+y})\leq \tau,& \forall\mathbf{s}\in\mathcal{S},\\
%\Longleftrightarrow & \mathbf{C}\bm{\tilde \pi}(\mathbf{s}) + \mathbf{D(Ys+y)} \leq \mathbf{d},\,  \mathbf{c}^\top\bm{\tilde\pi}(\mathbf{s})\leq \tau,&\forall\mathbf{s}\in\mathcal{S},
%\end{array}
%\end{equation}
where the first equivalence is due to statement $(i)$ in Lemma \ref{lem:liftingOperators}, and the second equivalence by definition of $\bm{\widetilde\pi}(\cdot)$. Hence, the quadruple $(\tau,\bm{\widetilde\pi}(\mathbf{\cdot}),\mathbf{Y},\mathbf{y})$ is feasible in \eqref{eq:compact_s}. Moreover, since $\mathcal{W}=\mathbf{Y}\mathcal{S}+\mathbf{y}$, problems \eqref{eq:compact_w} and \eqref{eq:compact_s} share the same objective function. Hence, both problems achieve the same objective value.

Conversely, suppose that $(\tau,\bm{\widetilde\pi}(\cdot),\mathbf{Y},\mathbf{y})$ is  feasible in \eqref{eq:compact_s}, and let $\mathcal{W}:=\mathbf{Y}\mathcal{S}+\mathbf{y}$. We now define $\bm\pi(\cdot)$ through $\bm\pi(\mathbf{w}):=\bm{\widetilde\pi}(L(\mathbf{w}))$, for all $\mathbf{w}\in\mathcal{W}$, where $L(\cdot)$ is a lifting operator satisfying Definition~\ref{def:Lifting}, for example the one given in \eqref{L}. By construction of $L(\cdot)$,  $\bm\pi(\cdot)$ remains causal, satisfying $\bm\pi(\cdot)\in\mathcal{C}$.
Since $\bm{\widetilde\pi}(\cdot)$ is feasible in \eqref{eq:compact_s},  we have that
\begin{equation*}
\begin{array}{rll}
        \vspace{0.1cm}
        & \hspace{-0.0cm} \left. \begin{array}{l}
        \mathbf{C}\bm{\widetilde\pi}(\mathbf{s}) + \mathbf{D(Ys+y)} \leq \mathbf{d}, \\[0.3ex]
        \mathbf{c}^\top\bm{\widetilde\pi}(\mathbf{s}) \leq \tau, 
        \end{array} \right\} \forall\mathbf{s}\in\mathcal{S}, \\[0.3ex]
        
        \vspace{0.1cm}
        \Longrightarrow & \hspace{-0.0cm} \left. \begin{array}{l}
         \mathbf{C}\bm{\widetilde\pi}(L(\mathbf{w})) + \mathbf{D}(\mathbf{Y}L(\mathbf{w})+\mathbf{y}) \leq \mathbf{d}, \\[0.3ex]
         \mathbf{c}^\top\bm{\widetilde\pi}(L(\mathbf{w})) \leq \tau, 
         \end{array} \right\} \forall \mathbf{w}\in\mathcal{W}, \\[0.3ex]
         
        \vspace{0.1cm}
        \Longleftrightarrow & \hspace{-0.0cm} \left. \begin{array}{l}
        \mathbf{C}\bm{\pi}(\mathbf{w}) + \mathbf{Dw} \leq \mathbf{d}, \\[0.3ex]
        \mathbf{c}^\top\bm{\pi}(\mathbf{w}) \leq \tau, 
        \end{array} \right\} \forall \mathbf{w}\in\mathcal{W},
\end{array}
\vspace{-0.0cm}
\end{equation*}
%\begin{equation}
%\begin{array}{rll}
%& \mathbf{C}\bm{\tilde\pi}(\mathbf{s}) + \mathbf{D(Ys+y)} \leq \mathbf{d},\, \mathbf{c}^\top\bm{\tilde\pi}(\mathbf{s}) \leq \tau, &\forall\mathbf{s}\in\mathcal{S},\\
%\Longrightarrow & \mathbf{C}\bm{\tilde\pi}(L(\mathbf{w})) + \mathbf{D}(\mathbf{Y}L(\mathbf{w})+\mathbf{y}) \leq \mathbf{d},\, \mathbf{c}^\top\bm{\tilde\pi}(L(\mathbf{w})) \leq \tau, &\forall \mathbf{w}\in\mathcal{W}(\mathbf{Y,y}),\\
%\Longleftrightarrow & \mathbf{C}\bm{\pi}(\mathbf{w}) + \mathbf{Dw} \leq \mathbf{d},\, \mathbf{c}^\top\bm{\pi}(\mathbf{w}) \leq \tau, &\forall \mathbf{w}\in\mathcal{W}(\mathbf{Y,y}),
%\end{array}
%\end{equation}
where the second line follows from statement $(ii)$ of Lemma~\ref{lem:liftingOperators}, and the third line follows from property (P2) in the definition of the lifting operator and the definition of $\bm\pi(\cdot)$. Hence, $(\tau,\mathcal{W},\bm{\pi}(\cdot),\mathbf{Y},\mathbf{y})$ is feasible in \eqref{eq:compact_w}. Since both problems share the same objective function due to $\mathcal{W}=\mathbf{Y}\mathcal{S}+\mathbf{y}$, both problems achieve the same objective value.
\end{proof}

\begin{proof}[Proof of Proposition~\ref{prop:relationshipPolicyW_and_PolicyS}]
To prove $(i)$, let $(\tau,\mathcal{W},\bm\pi(\cdot),\mathbf{Y,y})$ be a feasible solution to \eqref{eq:compact_w} with $\bm{\pi}(\cdot)\in\mathcal{C}_\text{aff}$.
Since $\bm\pi(\cdot)\in\mathcal{C}_\textnormal{aff}$, it can be written as $\bm\pi(\mathbf{w})=\mathbf{Q}\mathbf{w}+\mathbf{q}$ for some $\mathbf{Q}\in\R^{Nn_u\times N n_w}$ lower block-triangular, and $\mathbf{q}\in\R^{Nn_u}$. Now, let us  consider the candidate solution $(\tau,\bm{\widetilde\pi}(\cdot),\mathbf{Y,y})$ for \eqref{eq:compact_s}, with $\bm{\widetilde\pi}(\mathbf{s}) := (\mathbf{QY})\mathbf{s}+(\mathbf{Qy+q})$. Since $\mathbf{Y}$ is block-diagonal, the matrix $(\mathbf{QY})$ is also lower block-triangular, thus ensuring the causality of $\bm{\widetilde\pi}(\cdot)$. Therefore, $\bm{\widetilde\pi}(\mathbf{s}) := (\mathbf{QY})\mathbf{s}+(\mathbf{Qy+q})\in\mathcal{C}_\textnormal{aff}$. %It is clear that the same input is achieved by the two solutions. 
Moreover, since  $\bm{\pi}(\cdot)$ is feasible in \eqref{eq:compact_w}, we have that
\begin{equation*}
\begin{array}{lll}
        \vspace{0.1cm}
        & \hspace{-0.0cm} \left. \begin{array}{l}
        \mathbf{C}[\mathbf{Q}\mathbf{w}+\mathbf{q}] + \mathbf{Dw} \leq \mathbf{d}, \\[0.3ex]
        \mathbf{c}^\top[\mathbf{Q}\mathbf{w}+\mathbf{q}]\leq \tau, 
        \end{array} \right\} \forall\mathbf{w}\in\mathcal{W}, \\[0.3ex]
        
        \vspace{0.1cm}
        \Longleftrightarrow & \hspace{-0.0cm} \left. \begin{array}{l}
        \mathbf{C}[\mathbf{Q}(\mathbf{Ys+y})+\mathbf{q}] + \mathbf{D(Ys+y)} \leq \mathbf{d}, \\[0.3ex]
        \mathbf{c}^\top[ \mathbf{Q} (\mathbf{Ys+y})+\mathbf{q}]\leq \tau, 
        \end{array} \right\} \forall\mathbf{s}\in\mathcal{S},  \\[0.3ex]
        
        \vspace{0.1cm}
        \Longleftrightarrow & \hspace{-0.0cm} \left. \begin{array}{l}
        \mathbf{C} \bm{\widetilde\pi}(\mathbf{s}) + \mathbf{D}(\mathbf{Ys+y}) \leq \mathbf{d}, \\[0.3ex]
         \mathbf{c}^\top \bm{\widetilde\pi}(\mathbf{s}) \leq \tau, 
         \end{array} \right\} \forall\mathbf{s}\in\mathcal{S},
\end{array}
\end{equation*}
%\begin{equation*}
%\begin{array}{lll}
%&\mathbf{C}[\mathbf{Q}\mathbf{w}+\mathbf{q}] + \mathbf{Dw} \leq \mathbf{d},\, \mathbf{c}^\top[\mathbf{Q}\mathbf{w}+\mathbf{q}]\leq t,& \forall\mathbf{w}\in\mathcal{W}(\mathbf{Y,y}),\\
%\Longleftrightarrow &\mathbf{C}[\mathbf{Q}(\mathbf{Ys+y})+\mathbf{q}] + \mathbf{D(Ys+y)} \leq \mathbf{d},\,  \mathbf{c}^\top[ \mathbf{Q} (\mathbf{Ys+y})+\mathbf{q}]\leq t,&\forall\mathbf{s}\in\mathcal{S} , \\
%        \Longleftrightarrow & \mathbf{C}[(\mathbf{QY})\mathbf{s} + (\mathbf{Qy}+\mathbf{q})] + \mathbf{D}(\mathbf{Ys+y}) \leq \mathbf{d},  \mathbf{c}^\top[(\mathbf{QY})\mathbf{s}+(\mathbf{Qy}+\mathbf{q})]\leq t,& \forall\mathbf{s}\in\mathcal{S},
%\end{array}
%\end{equation*}
where the first equivalence is due to statement $(i)$ in Lemma \ref{lem:liftingOperators}, and the second equivalence follows from the definition of $\bm{\widetilde\pi}(\cdot)$. Hence, $(\tau,\bm{\widetilde\pi}(\cdot),\mathbf{Y,y})$ is a feasible solution for \eqref{eq:compact_s}, achieving the same objective value.

\blue{
To prove $(ii)$, let $(\tau,\bm{\widetilde\pi}(\cdot),\mathbf{Y,y})$ be a feasible solution to \eqref{eq:compact_s} with  $\bm{\widetilde\pi}(\cdot)\in\mathcal{C}_\text{aff}$ and $\mathbf Y$ invertible, and define $\mathcal{W}:=\mathbf{Y}\mathcal{S}+\mathbf{y}$. Since $\bm{\widetilde\pi}(\cdot)\in\mathcal{C}_\text{aff}$, it can be written as $\bm{\widetilde\pi}(\bs{s}) = \mathbf{Ps+p}$ for some $\mathbf{P}\in\R^{Nn_u\times N n_s}$ lower block-triangular, and $\mathbf{p}\in\R^{Nn_u}$. Consider now the candidate solution $(\tau,\mathcal{W},\bm{\pi}(\cdot),\mathbf{Y,y})$ for \eqref{eq:compact_w}, with $\bm{\pi}(\mathbf{w}) := \mathbf P \left[ \mathbf Y^{-1} (\mathbf w-\mathbf y)  \right] + \mathbf p$, for all $\mathbf w\in\mathcal W$. It is easy to see that $\bm\pi(\cdot)$ is an affine policy. Moreover, since $\mathbf Y$ is block-diagonal by construction, see \eqref{eq:TrafoS2W}, $\bm\pi(\cdot)$ is also causal, hence satisfying $\bm\pi(\cdot)\in\mathcal C_\text{aff}$. Finally, since $\bm{\widetilde\pi}(\cdot)$ is feasible in \eqref{eq:compact_s}, we have that
}
\begin{equation*}
\begin{array}{rll}
        \vspace{0.1cm}
        & \left. \begin{array}{l}
        \mathbf{C}[\mathbf{Ps+p}] + \mathbf{D(Ys+y)} \leq \mathbf{d}, \\ [0.3ex]
        \mathbf{c}^\top [\mathbf{Ps+p}] \leq \tau , 
        \end{array} \right\} \forall\mathbf{s}\in\mathcal{S},  \\[0.3ex]
        
        \vspace{0.1cm}
        \Longleftrightarrow &  \left. \begin{array}{l}
        \mathbf{C} [\mathbf{P}\mathbf Y^{-1}(\mathbf w - \mathbf y)+\mathbf{p}] + \mathbf{Dw} \leq \mathbf{d}, \\ [0.3ex]
        \mathbf{c}^\top [\mathbf{P}\mathbf Y^{-1}(\mathbf w - \mathbf y)+\mathbf{p} ] \leq \tau, 
        \end{array} \right\} \forall \mathbf{w} \in\mathcal{W},  \\[0.3ex]
        
        \vspace{0.1cm}
        \Longleftrightarrow & \left. \begin{array}{l}
        \mathbf{C} \bm\pi(\mathbf w) + \mathbf{D}\mathbf{w} \leq \mathbf{d}, \\ [0.3ex]
        \mathbf{c}^\top \bm\pi(\mathbf w)  \leq \tau, 
        \end{array} \right\} \forall \mathbf{w}\in\mathcal{W},
\end{array}
\end{equation*}
\blue{
where the first equivalence follows from statement $(i)$ of Lemma~\ref{lem:liftingOperators} and the fact that $\mathbf w = \mathbf Y \mathbf s + \mathbf y$ is equal to $\mathbf s = \mathbf Y^{-1} ( \mathbf w - \mathbf y )$, and the second equivalence follows from the definition of $\bm\pi(\cdot)$.  Hence, $(\tau,\mathcal{W},\bm{\pi}(\cdot),\mathbf{Y,y})$ is a feasible solution for \eqref{eq:compact_w}, and achieves the same objective value.
}

To prove $(iii)$,  let $(\tau,\bm{\widetilde\pi}(\cdot),\mathbf{Y,y})$ be a feasible solution to \eqref{eq:compact_s} with  $\bm{\widetilde\pi}(\cdot)\in\mathcal{C}_\text{aff}$, and define $\mathcal{W}:=\mathbf{Y}\mathcal{S}+\mathbf{y}$. Since $\bm{\widetilde\pi}(\cdot)\in\mathcal{C}_\text{aff}$, it can be written as $\bm{\widetilde\pi}(\bs{s}) = \mathbf{Ps+p}$ for some $\mathbf{P}\in\R^{Nn_u\times N n_s}$ lower block-triangular, and $\mathbf{p}\in\R^{Nn_u}$. Let $L(\cdot)=[L_0(\cdot),\ldots,L_{N-1}(\cdot)]$ be the lifting operator with its elements $L_k(\cdot)$ defined as in \eqref{L}, and  consider the candidate solution $(\tau,\mathcal{W},\bm{\pi}(\cdot),\mathbf{Y,y})$ for \eqref{eq:compact_w}, with $\bm{\pi}(\mathbf{w}) := \mathbf{P}L(\mathbf{w})+\mathbf{p}$, for all $\mathbf w\in\mathcal W$. Under the assumption of $\mathcal{S}$ being a polytope, $L(\cdot)$ is a continuous piece-wise affine  function defined through the solution of a parametric quadratic optimization problem \cite[Theorem 7.7]{BorBemMorari_book}. Moreover, since  $L(\cdot)$ satisfies  properties \textnormal{(P1)} and \textnormal{(P2)} in Definition~\ref{def:Lifting}, it preserves causality, and $\bm{\pi}(\mathbf{w}) := \mathbf{P}L(\mathbf{w})+\mathbf{p}\in\mathcal{C}_\text{pwa}$  is a causal policy. Finally, since $\bm{\widetilde\pi}(\cdot)$ is feasible in \eqref{eq:compact_s}, we have that
\begin{equation*}
\begin{array}{rll}
        \vspace{0.1cm}
        & \left. \begin{array}{l}
        \mathbf{C}(\mathbf{Ps+p}) + \mathbf{D(Ys+y)} \leq \mathbf{d}, \\ [0.3ex]
        \mathbf{c}^\top(\mathbf{Ps+p}) \leq \tau , 
        \end{array} \right\} \forall\mathbf{s}\in\mathcal{S},  \\[0.3ex]
        
        \vspace{0.1cm}
        \Longrightarrow &  \left. \begin{array}{l}
        \mathbf{C}(\mathbf{P}L(\mathbf{w})+\mathbf{p}) + \mathbf{D(Y}L(\mathbf{w})+\mathbf{y}) \leq \mathbf{d}, \\ [0.3ex]
        \mathbf{c}^\top(\mathbf{P}L(\mathbf{w})+\mathbf{p}) \leq \tau, 
        \end{array} \right\} \forall \mathbf{w} \in\mathcal{W},  \\[0.3ex]
        
        \vspace{0.1cm}
        \Longleftrightarrow & \left. \begin{array}{l}
        \mathbf{C}(\mathbf{P}L(\mathbf{w})+\mathbf{p}) + \mathbf{D}\mathbf{w} \leq \mathbf{d}, \\ [0.3ex]
        \mathbf{c}^\top(\mathbf{P}L(\mathbf{w})+\mathbf{p}) \leq \tau, 
        \end{array} \right\} \forall \mathbf{w}\in\mathcal{W},
\end{array}
\end{equation*}
%\begin{equation}
%\begin{array}{rll}
%%& \mathbf{C}\bm{\tilde\pi}(\mathbf{s}) + \mathbf{D(Ys+y)} \leq \mathbf{d},\, \mathbf{c}^\top\bm{\tilde\pi}(\mathbf{s}) \leq \tau, &\forall\mathbf{s}\in\mathcal{S},\\
%& \mathbf{C}(\mathbf{Ps+p}) + \mathbf{D(Ys+y)} \leq \mathbf{d},\, \mathbf{c}^\top(\mathbf{Ps+p}) \leq \tau, &\forall\mathbf{s}\in\mathcal{S},\\
%\Longrightarrow & \mathbf{C}(\mathbf{P}L(\mathbf{w})+\mathbf{p}) + \mathbf{D(Y}L(\mathbf{w})+\mathbf{y}) \leq \mathbf{d},\, \mathbf{c}^\top(\mathbf{P}L(\mathbf{w})+\mathbf{p}) \leq \tau, &\forall \mathbf{w}\in\mathcal{W}(\mathbf{Y,y}),\\
%\Longleftrightarrow & \mathbf{C}(\mathbf{P}L(\mathbf{w})+\mathbf{p}) + \mathbf{D}\mathbf{w} \leq \mathbf{d},\, \mathbf{c}^\top(\mathbf{P}L(\mathbf{w})+\mathbf{p}) \leq \tau, &\forall \mathbf{w}\in\mathcal{W}(\mathbf{Y,y}),
%\end{array}
%\end{equation}
where the second line follows from statement $(ii)$ of Lemma~\ref{lem:liftingOperators}, and the third line follows from property (P3).
 Hence, $(\tau,\mathcal{W},\bm{\pi}(\cdot),\mathbf{Y,y})$ is a feasible solution for \eqref{eq:compact_w}, and achieves the same objective value.
\end{proof}

\end{appendix}

%\newpage
\bibliographystyle{unsrt}
\bibliography{library}

\end{document}